\documentclass[12pt]{amsart} 
\usepackage[left=3cm,top=2.5cm,bottom=2.5cm,right=3cm]{geometry}
\usepackage{times}
\usepackage{amssymb, amsmath, amsthm}
\usepackage{mathtools}
\usepackage{graphicx,xspace}
\usepackage{epsfig}
\usepackage{enumitem}
\usepackage[usenames,dvipsnames]{xcolor}
\usepackage{tikz}
\usepackage{gensymb}
\usepackage[T1]{fontenc}
\usepackage[utf8]{inputenc}
\usepackage{bm}
\usepackage{subcaption}
\usepackage{mathrsfs}
\definecolor{green}{RGB}{0,127,0}
\definecolor{red}{RGB}{191,0,0}
\usepackage[colorlinks,cite color=red,link color=green,pagebackref=true]{hyperref}
\usepackage{todonotes}

\newcommand{\todoValentin}[1]{\todo[color=pink!40]{\textsc{Valentin} says: #1}}
\usetikzlibrary{matrix,arrows,calc}
\usetikzlibrary{decorations.markings}
\usetikzlibrary{patterns}
\usetikzlibrary{snakes}
\usetikzlibrary{shapes}

\usepackage[capitalize]{cleveref}
\usepackage{seqsplit}
\usepackage{booktabs}

\theoremstyle{plain}

\newtheorem{lemma}{Lemma}[section]
\newtheorem{theorem}[lemma]{Theorem}
\newtheorem{corollary}[lemma]{Corollary}
\newtheorem{proposition}[lemma]{Proposition}

\newtheorem{defprop}[lemma]{Definition-Proposition}
\newtheorem{definition}[lemma]{Definition}
\newtheorem{definition-lemma}[lemma]{Definition-Lemma}

\theoremstyle{remark}
\newtheorem{remark}{Remark}

\newcommand{\xdownarrow}[1]{{\left\downarrow\vbox to #1{}\right.\kern-\nulldelimiterspace}}
\newcommand{\splus}{\!+\!}
\newcommand{\sminus}{\!-\!}

\newcommand{\RR}{\mathbb{R}}
\newcommand{\C}{\mathbb{C}}

\newcommand{\Fth}{F^{\theta}}
\newcommand{\Ftr}{F^{\theta_3}}
\newcommand{\Gth}{G^{\theta}}

\newcommand{\Sym}[1]{\mathfrak{S}_{#1}}

\newcommand{\Z}{\mathbb{Z}}

\newcommand{\QQ}{\mathbb{Q}}

\newcommand{\pp}{\mathbf{p}}
\newcommand{\qq}{\mathbf{q}}

\newcommand{\PPP}{\mathcal{P}}

\DeclareMathOperator{\Pf}{Pf}
\def\la{\lambda}

\def\a{\alpha}

\DeclareMathOperator{\BGW}{BGW}

\DeclareMathOperator{\Tr}{Tr}
\DeclareMathOperator{\Mat}{Mat}

\DeclareMathOperator{\sgn}{sgn}

\DeclareMathOperator{\Gl}{GL}

\newcommand{\m}{\mathfrak{h}}

\DeclareMathOperator{\Symm}{Sym}

\DeclareMathOperator{\hook}{hook}
\DeclareMathOperator{\Res}{Res}
\DeclareMathOperator{\dd}{d}

\newcommand{\tttt}{\mathfrak{t}}
\newcommand{\hh}{\mathfrak{h}}

\newcommand{\hsum}{\mathop{\sum\hspace{-5mm}\Big/}}

\title[Enumeration of non-oriented maps via integrability]{Enumeration of non-oriented maps via integrability}

\author[V.~Bonzom]{Valentin Bonzom}
\address{Universit\'e Sorbonne Paris Nord, LIPN, CNRS, UMR 7030, F-93430 Villetaneuse, France}
\email{bonzom@lipn.univ-paris13.fr}

\author[G.~Chapuy]{Guillaume Chapuy}
\address{CNRS, IRIF UMR 8243, Universit\'e de Paris.}
\email{guillaume.chapuy@irif.fr}

\author[M.~Dołęga]{Maciej Dołęga}
\address{
Institute of Mathematics, 
Polish Academy of Sciences, 
ul. Śniadeckich 8, 
00-956 Warszawa, Poland.
}
\email{mdolega@impan.pl}

\thanks{This project has received funding from the European Research
  Council (ERC) under the European Union’s Horizon 2020 research and
  innovation programme (grant agreement No. ERC-2016-STG 716083
  “CombiTop”). MD is supported from {\it Narodowe Centrum Nauki},
  grant UMO-2017/26/D/ST1/00186. VB is partially supported by the ANR-20-CE48-0018 3DMaps.}

\begin{document}
\begin{abstract}

In this note, we examine how the BKP structure of the generating series of several models of maps on non-oriented surfaces can be used to obtain explicit and/or efficient recurrence formulas for their enumeration according to the genus and size parameters.

	Using techniques already known in the orientable case (elimination of variables via Virasoro constraints or Tutte equations), we naturally obtain recurrence formulas with non-polynomial coefficients. This non-polynomiality reflects the presence of shifts of the charge parameter in the BKP equation. Nevertheless, we show that it is possible to obtain non-shifted versions, meaning pure ODEs for the associated generating functions, from which recurrence relations with polynomial coefficients can be extracted. We treat the cases of triangulations, general maps, and bipartite maps.

	These recurrences with polynomial coefficients are conceptually interesting but bigger to write than those with non-polynomial coefficients. However they are relatively nice-looking in the case of one-face maps. In particular we show that Ledoux's recurrence for non-oriented one-face maps can be recovered in this way, and we obtain the analogous statement for the (bivariate) bipartite case.
      \end{abstract}

\maketitle


\section{Introduction and main results}

In this note, we are interested in obtaining simple, or at least efficient, recurrence formulas to count maps on surfaces according to their genus and  size parameters. 
For us, a \emph{map} is the 2-cell embedding of a connected multigraph in a compact connected surface, considered up to homeomorphism. Our surfaces are not necessarily orientable, and we call \emph{genus} of a surface the number $g\in \tfrac{1}{2}\mathbb{N}$ such that its Euler characteric is $2-2g$. The sphere has genus $0$, the projective plane has genus $\tfrac{1}{2}$, the torus and Klein bottle have genus $1$, etc.

Perhaps one of the nicest-looking formulas in the field of map enumeration is the Goulden-Jackson recurrence formula for orientable \emph{triangulations}, i.e. maps in which all faces are incident to three edge-sides. The Goulden-Jackson recurrence~\cite{GouldenJackson2008}, in fact also discovered in an equivalent form  in~\cite[Eq.~(B.6)]{KazakovKostovNekrasov1999}, asserts that the number $t_{n,g}$ of rooted\footnote{precise definitions all terms used in this introduction are given in the later sections.} triangulations   with $n$ faces on an \emph{orientable}  surface of genus $g$  is solution of the equation 
\begin{align}\label{eq:GJ}
	(n+1)	t_{n,g} = 4 n(3n-2) (3n-4) t_{n-1,g-1}+4\sum_{i+j=n-2\atop h+k=g}
	(3i+2)(3j+2)t_{i,h} t_{j,k}.
\end{align}
This formula was immediately recognized as a breakthrough in the field, 
because it gives a much better access to these numbers (computational or theoretical) than the classical techniques.

Indeed, in the classical approach, one introduces generating functions of maps of genus $g$ with a certain number of additional boundaries, and one shows that a combinatorial operation of root-deletion on the maps (the ``Tutte decomposition'') implies a functional equation for these functions. This approach has been very successful in the planar case since the work of Tutte, see e.g.~\cite{Tutte1962a,Tutte1962c,Tutte1963,BenderCanfield1994,BernardiBousquetMelou2017,BernardiBousquetMelou2011}. In higher genus, it was pioneered by 
Lehman and Walsh~\cite{WalshLehman1972} 
and later Bender and Canfield,
who showed that the generating functions of maps of fixed genus and number of  boundaries can be computed inductively on the Euler characteristic, thus revealing their particularly nice algebraic structure as well as their singular behaviour~\cite{BenderCanfield1986, BenderCanfield1991}. 
Bender and Canfield's inductive technique can be seen as a predecessor of the Chekhov--Eynard--Orantin
topological recursion~\cite{ChekhovEynardOrantin2006,EynardOrantin2007}, a powerful theory invented in the context of matrix integrals~\cite{Eynard:book} which has now been applied to study the structure of fixed-genus generating functions of many models of maps or in enumerative geometry~\cite{Eynard2014,AlexandrovChapuyEynardHarnad2020,BychkovDuninBarkowskiKazarianShadrin2020,BelliardCharbonnierEynardGarciaFailde2021}.

The need to introduce additional boundaries (and ``catalytic'' variables to mark their sizes) makes these approaches ineffective  for large values of $g$. There seems to be no hope to obtain control on the bivariate numbers $t_{n,g}$ for non-fixed $g$ in this way, a striking contrast with the recurrence~\eqref{eq:GJ}. For example~\eqref{eq:GJ} also gives access to the so-called \emph{double-scaling limit} of the numbers $t_{n,g}$ \cite{KazakovKostovNekrasov1999,BenderGaoRichmond2008}, and it is also crucially used in the recent Budzinski-Louf breakthrough on large genus asymptotics~\cite{BudzinskiLouf2020a}. 
Perhaps we should insist on the fact that we mean no harm to the ``classical approach''. The study of the rational parametrization it gives rise to is a fascinating subject, including purely bijective combinatorics~\cite{ChapuyMarcusSchaeffer2009,ChapuyDolega2017,Lepoutre2019,AlbenqueLepoutre2019,DolegaLepoutre2020}, with probable link to the study of random geometries~\cite{BouttierGuitter2014, BouttierGuitterMiermont2021}.
On the other hand, as of today, the  bijective interpretation of the Goulden-Jackson recurrence~\eqref{eq:GJ} is wide open.

\medskip 

One reason why the recurrence~\eqref{eq:GJ} gives access to different results is because it comes from a completely different technique.
It is based on the fact that the generating function of maps on orientable surfaces, with an infinite number of variables $p_i, i\geq 1$ ($p_i$ marking faces of degree $i$), is a solution of the KP hierarchy -- an infinite sequence of PDEs originating from the theory of integrable systems, with deep connections to infinite dimensional Lie algebras and algebraic combinatorics~\cite{Kac2013Bombay,MiwaJimboDate2000}. The first equation of the hierarchy (\emph{the KP equation}) reads
\begin{align}\label{eq:KP}
-F_{3,1} + F_{2^2} + \frac{1}{2} F_{1^2}^2 + \frac{1}{12} F_{1^4} = 0,
\end{align}
where each $i$-index indicates a partial derivative with respect to $p_i$. In order to go from the KP equation~\eqref{eq:KP} to the recurrence~\eqref{eq:GJ}, Goulden and Jackson use the fact that the generating function $F(p_1,p_2,p_3,0,\dots)$ of maps having only faces of sizes $1,2,3$ can in fact be expressed in terms of  the series $F(0,0,t,0,\dots)$ of triangulations only. This enables one so set $p_i=t \delta_{i,3}$ in \eqref{eq:KP} and obtain an \emph{ODE} for the generating function of triangulations. The fact that the variables $p_1$ and $p_2$ can be \emph{eliminated} in this way relies on local surgery operations that can, in fact, be interpreted as first cases of the classical Tutte decomposition.

A similar elimination technique has since been used to obtain similar results for other models of maps. In~\cite{CarrellChapuy2015}, Carrell and the second author use the fact that the generating function of bipartite maps solves the KP hierarchy, and local operations related to the first Tutte equations for bipartite quadrangulations, to obtain a recurrence formula similar to~\eqref{eq:GJ} to count maps by vertices and faces. In~\cite{KazarianZograf2015}, Kazarian and Zograf use a slightly different elimination procedure, using the so-called Virasoro constraints (which are also related to Tutte decompositions) to recover the recurrence of~\cite{CarrellChapuy2015} and to obtain an analogue for bipartite maps.
These three works only use the first KP equation.
Finally, Louf~\cite{Louf2019} uses a different integrable hierarchy (the Toda hierarchy) to obtain a remarkable recurrence counting bipartite maps of arbitrary genus with control on all face degrees, using a different elimination technique inspired by Okounkov's work on Hurwitz numbers~\cite{Okounkov2000a}.

\medskip
In this paper, we are interested in obtaining variants of these results for non-oriented surfaces. Our starting point is the fact that generating functions of maps (or bipartite maps) are solutions of the BKP hierarchy of Kac and Van De Leur~\cite{KacVandeLeur1998} (see also the appendix of \cite{BonzomChapuyDolega2021}). An important difference between the KP and BKP hierarchy is that the function $F$ which is a solution of this hierarchy also involves a so-called \emph{charge parameter} $N$, which in our context will always be a variable marking faces or vertices of a certain kind. The first BKP equation reads;
\begin{equation}
    \label{BKP1Log}
\sminus{}F_{3,1}(N) \splus{} F_{2^2}(N) \splus{} \tfrac{1}{2} F_{1^2}(N)^2 \splus{} \tfrac{1}{12} F_{1^4}(N)=
	 S_2(N) \tau(N\!\sminus{}\!2) \tau(N\!\splus{}\!2)\tau(N)^{-2},
\end{equation}
where $\tau(N)=e^{F(N)}$ and where $S_2(N)$ is a model-dependant normalizing factor that will always be an explicit rational function in our case. 
In~\cite{Carrell2014}, Carrell used the fact that the generating function of non-oriented maps satisfies this equation, together with the elimination techniques developed by Goulden and Jackson in the orientable case, to obtain a functional equation for the case of triangulations (this technique leads to an explicit recurrence, see Theorem~\ref{thm:TriangulationsRecurrence} below).

The first task we perform in this paper, somehow unsurprisingly, is to apply the elimination of variables from the papers ~\cite{CarrellChapuy2015,KazarianZograf2015} to the BKP equation, to obtain recurrences of the same kind to count maps (by vertices and edges) and bipartite maps (by edges, and vertices of each colour) on non-oriented surfaces. The Virasoro constraints for these models are known (e.g.~\cite{BonzomChapuyDolega2021}) and our main task here is to make sure that the elimination procedure indeed works, i.e. that these equations indeed enable to reduce all derivatives appearing in~\eqref{BKP1Log} to differential polynomials in a single variable.
For completeness, we also treat Carrell's case of triangulations explicitly.
All these recurrences are larger than~\eqref{eq:GJ}, but incredibly short compared to any alternative, and it is not unreasonable to believe that they could have a combinatorial interpretation. For example, we obtain in Section~\ref{sec:alaCC15} the following recurrence formula.
Everywhere in the paper, the symbol $\hsum$ denotes a sum over elements of $\tfrac{1}{2}\mathbb{N}$.
\begin{theorem}[Counting maps by edges and genus]\label{thm:recurrence2}
	The number $\hh_n^g$ of rooted maps of genus $g$ with $n$ edges, orientable or not, can be computed from the following recurrence formula:
	\begin{multline} \label{eq:recurrenceCC}
	\hh_n^g
=
\tfrac{2}{(n+1)(n-2)} \Bigg(
n(2n-1)( 2 \hh_{n-1}^{g} + \hh_{n-1}^{g-1/2})
	+   \tfrac{(2 n-3) (2 n-2) (2 n-1) (2 n)}{2} \hh_{n-2}^{g-1}
\\ 
+12 \mkern-10mu\hsum_{g_1=0..g\atop g_1+g_2=g} \sum_{n_1=0..n \atop n_1+n_2=n} 
\mkern-10mu\tfrac{(2 n_2-1)(2 n_1-1)n_1}{2}        \hh_{n_2-1}^{g_2}          
 \hh_{n_1-1}^{g_1}   
-\mkern-10mu
\hsum_{g_1=0..g\atop g_1+g_2=g} \sum_{n_1=0..n-1\atop n_1+n_2=n}
	 \hsum_{g_0=0..g_1\atop g_1-g_0\in \mathbb{N}} \mkern-10mu {\textstyle {n_1+2-2 g_0 \choose n_1-2 g_1}} 
 2^{2 (1+g_1-g_0)}
      \hh_{n_1}^{g_0}
      \\
 \Big(
	 \tfrac{(2 n_2-1) (2 n_2-2) (2 n_2-3)}{2} \hh_{n_2-2}^{g_2-1}
	 -\delta_{(n_2,g_2)\neq(n,g)}\tfrac{n_2+1}{4} {\hh}_{n_2}^{g_2} 
 +\tfrac{2 n_2-1}{2} ( 2 \hh_{n_2-1}^{g_2} + \hh_{n_2-1}^{g_2-1/2} )\\
+6 
\hsum_{g_3=0..g_2\atop g_3+g_4=g_2}
\sum_{n_3=0..n_2\atop n_3+n_4=n_2}
\tfrac{(2 n_3-1)(2 n_4-1) }{4}\hh_{n_3-1}^{g_3} \hh_{n_4-1}^{g_4}
\Big)\vspace{3cm} \Bigg)	
	\end{multline}
	for $n>2$, with the initial conditions 
	$\hh_0^0=1$, $\hh_1^0=2$, $\hh_2^0=9$, $\hh_{1}^{1/2}=1$, $\hh_2^{1/2}=10$, $\hh_2^1=5$, and $\hh_n^g=0$ if $n<2g$.
	\end{theorem}
We will obtain similar theorems for other models, in particular one with control on vertices and faces (Theorem~\ref{thm:recurrence2bis} or Theorem~\ref{thm:MapsRecurrence} depending on the elimination technique), one for bipartite maps (Theorem~\ref{thm:BipMapsRecurrence}),  and one for triangulations which is implicit in Carrel's work (Theorem~\ref{thm:TriangulationsRecurrence}).

The crucial fact that the BKP equation~\eqref{BKP1Log} involves not only the function $F(N)$ but also its shifts $F(N+2)$ and $F(N-2)$ has an important effect on the recurrence formulas we obtain.
The functional equations corresponding to these recurrences, which involve derivatives but also shifts of variables, are \emph{not} ODE in their main variable.
In return, the recurrences obtained do not have polynomial coefficients (for example \eqref{eq:recurrenceCC} contains binomial coefficients, which are not polynomials in the summation variables). This is a deep structural difference between the recurrences~\eqref{eq:GJ} and recurrences such as~\eqref{eq:recurrenceCC}.

\medskip

It is natural to ask if one could instead obtain  formulas in which the shifts are not involved, i.e. true polynomial recurrence formulas, corresponding to nonlinear ODEs with polynomial coefficients for the associated generating functions. This would be much more satisfying, at least at the conceptual level.
Maybe surprisingly, we will see that the answer to this question is \emph{yes}.
To see this, we will have to use several (in fact, three) equations of the BKP hierarchy. Using additional derivations and manipulations, we will be able to eliminate the shifts from equations, and obtain equations at fixed $N$, at the price of having to consider higher derivatives. It is not obvious, but it will be true, that a finite number of Virasoro constraints will still be sufficient to perform the elimination of variables in this context.

Due to the use of higher BKP equations and additional manipulations involved, the equations thus obtained are bigger than the previous ones\footnote{They are however not gigantic, see the accompanying worksheet~\cite{us:Maple}. The main ODE for maps fits in slightly more than a page in $\backslash$tiny LaTex print, we have however chosen not to reproduce it here.}. We will only state them here in a non-explicit form. The reader eager to see them at work may access these equations, and use them to compute numbers of maps, in the accompanying Maple worksheet~\cite{us:Maple}.
A typical statement we obtain from these methods is the following.
\begin{theorem}[Counting maps by edges and genus -- unshifted recurrence]\label{thm:rec-maps-nonshifted}
The number $\hh_n^g$ of rooted maps of genus $g$ with $n$ edges, orientable or not, 
is solution of an explicit recurrence relation of the form
	\begin{align}\label{eq:rec-maps-nonshifted}
         \hh_n^g = 
		\sum_{a=0}^{K_1}\hsum_{b=0}^{K_2} \sum_{k=1}^{K_3}
            \sum_{n_1,\dots,n_k \geq 1 \atop n_1+\cdots + n_k =n-a}
            \hsum_{g_1,\dots,g_k \geq 0 \atop  g_1+\dots + g_k =g-b}
	P_{a,b,k}(n_1,\dots,n_k) \hh_{n_1}^{g_1} \hh_{n_2}^{g_2}\dots \mathfrak{t}_{n_k}^{g_k},
	\end{align}
	where the $P_{a,b,k}$ are rational functions with $P_{0,0,1}=0$, and $K_1,K_2,K_3 <\infty$.
\end{theorem}
We will obtain similar theorems for other models, in particular 
one for bipartite maps (Theorem~\ref{theo:ODEBipMaps}), and for triangulations  (Theorem~\ref{theo:ODETriangulations}).
Moreover, we will in fact prove a version of Theorem~\ref{thm:rec-maps-nonshifted} with control on the number of faces, from which we obtain a closed recurrence formula enumerating one-face maps, small enough to be explicitly written.
\begin{theorem}[Ledoux's recursion for non-oriented one-face maps]
 The number $\mathfrak{u}_{n}^g$ of rooted non-oriented maps
     of genus $g$ with $n$ edges and only one face (or equivalently
     with $n$ edges and only one vertex) is given by the recursion
     \begin{multline}
       (n+1)\mathfrak{u}_{n}^g = (8n-2) \mathfrak{u}_{n-1}^g-(4n-1) \mathfrak{u}_{n-1}^{g-1/2}
                                 + n(2n-3)(10n-9)
                                 \mathfrak{u}_{n-2}^{g-1} - 8(2n-3) \mathfrak{u}_{n-2}^g 
\\
-10(2n-3)(2n- 4)(2n-5) \mathfrak{u}_{n-3}^{g-1} 
+ 5(2n-3)(2n-4)(2n-5) \mathfrak{u}_{n-3}^{g-3/2} \\
+ 8(2n-3) \mathfrak{u}_{n-2}^{g-1/2} - 2(2n-3)(2n-4)(2n-5)(2n-6)(2n-7) \mathfrak{u}_{n-4}^{g-2} \label{eq:Ledoux}
     \end{multline}
with the convention that $\mathfrak{u}_{n}^{g} = 0$ for $g<0$ and $g
>\frac{n}{2}$ and with the initial condition
$\mathfrak{u}_{1}^{1/2}=1$,
$\mathfrak{u}_{2}^{1}=\mathfrak{u}_{2}^{1/2} = 5$,
$\mathfrak{u}_{3}^{3/2}=41$, $\mathfrak{u}_{3}^{1}=52$, $\mathfrak{u}_{3}^{1/2}=22$.
\end{theorem}
The recurrence~\eqref{eq:Ledoux} was first obtained by Ledoux~\cite{Ledoux2009} using matrix integral techniques unrelated (as far as we know) to the BKP equation. It is remarkable to see that it is, in fact, the shadow of bigger nonlinear recurrence giving access to an arbitrary number of faces. The Ledoux recurrence can be viewed as an non-oriented version of the infamous Harer-Zagier recurrence, a similar (yet smaller) formula which covers the case of orientable one-face maps (and which is itself a special case of the recurrence of~\cite{CarrellChapuy2015}). The Harer-Zagier recurrence has a nice analogue in the bipartite case due to Adrianov~\cite{Adrianov1997}, and it is natural to ask if our non-shifted recursion in the bipartite case implies an non-oriented version of Adrianov's result. The answer is yes.
\begin{theorem}[A recurrence for non-oriented bipartite one-face maps]
	The number $\mathfrak{b}_{n}^{i,j}$ of
rooted one-face maps with $n$ edges, $i$ white and $j$ black
vertices, orientable or not, is given by the recursion:
\begin{align}
	(n\splus{}1) \mathfrak{b}_{n}^{i,j} = &\hphantom{+} 
	(4n\sminus{}1)(\mathfrak{b}_{n\sminus{}1}^{i-1,j}\splus{}\mathfrak{b}_{n\sminus{}1}^{i,j-1}\sminus{}\mathfrak{b}_{n\sminus{}1}^{i,j})
	\splus{}(5n^3 \sminus{} 16n^2 \splus{} 13n\sminus{} 1) \mathfrak{b}_{n\sminus{}2}^{i,j}
	\nonumber\\&
\splus{}(2n \sminus{} 3)(4\mathfrak{b}_{n\sminus{}2}^{i-1,j} \splus{} 4\mathfrak{b}_{n\sminus{}2}^{i,j-1} \sminus{}3\mathfrak{b}_{n\sminus{}2}^{i-2,j} \splus{} 3\mathfrak{b}_{n\sminus{}2}^{i,j-2}\sminus{} 2 \mathfrak{b}_{n\sminus{}2}^{i-1,j-1}) 
	\nonumber\\&
\splus{}(10n^3 \sminus{} 68n^2 \splus{} 150n \sminus{}107)(\mathfrak{b}_{n\sminus{}3}^{i,j}\sminus{}\mathfrak{b}_{n\sminus{}3}^{i-1,j}\sminus{}\mathfrak{b}_{n\sminus{}3}^{i,j-1})
	\nonumber\\&
\splus{}(4n	\sminus{}11)(\mathfrak{b}_{n\sminus{}3}^{i-3,j}
\splus{}\mathfrak{b}_{n\sminus{}3}^{i,j-3}\sminus{}\mathfrak{b}_{n\sminus{}3}^{i-2,j-1}
\sminus{}\mathfrak{b}_{n\sminus{}3}^{i-1,j-2} \sminus{}
\mathfrak{b}_{n\sminus{}3}^{i-2,j}\sminus{}\mathfrak{b}_{n\sminus{}3}^{i,j-2} \splus{} 2
	\mathfrak{b}_{n\sminus{}3}^{i-1,j-1})\nonumber\\&
\splus{}(4\sminus{}n)((2n \sminus{} 7)^2(n \sminus{} 2)^2 \mathfrak{b}_{n\sminus{}4}^{i,j}\splus{}(5n^2 \sminus{} 32n \splus{}
53)(\mathfrak{b}_{n\sminus{}4}^{i-2,j}\splus{}\mathfrak{b}_{n\sminus{}4}^{i,j-2}\sminus{}2
	\mathfrak{b}_{n\sminus{}4}^{i-1,j-1})\nonumber\\&
	\splus{}\mathfrak{b}_{n\sminus{}4}^{i-4,j}\splus{}\mathfrak{b}_{n\sminus{}4}^{i,j-4}\sminus{}4 \mathfrak{b}_{n\sminus{}4}^{i-3,j-1}\splus{}4\mathfrak{b}_{n\sminus{}4}^{i-1,j-3}\splus{}6 \mathfrak{b}_{n\sminus{}4}^{i-2,j-2})) \label{eq:adrianovnori}
\end{align}
with the convention that $\mathfrak{b}_{n}^{i,j} =0$ for
$i+j>n+1$, and $\mathfrak{b}_{n}^{i,0}=\mathfrak{b}_{n}^{0,j}=0$ and the initial conditions $\mathfrak{b}_{1}^{1,1} =
\mathfrak{b}_{2}^{2,1} = \mathfrak{b}_{2}^{1,2} =
\mathfrak{b}_{2}^{1,1}= 1$,
$\mathfrak{b}_{3}^{3,1}=\mathfrak{b}_{3}^{1,3}=1$,
$\mathfrak{b}_{3}^{2,2}=\mathfrak{b}_{3}^{2,1}=\mathfrak{b}_{3}^{1,2}=3$,
 $\mathfrak{b}_{3}^{1,1}=4$.
\end{theorem}

To conclude this introduction, it is natural to ask if our techniques of shift elimination are specific to the case of maps or apply to general solutions of the BKP hierarchy. The latter is in fact true, and any function $F(N)$ which solves the BKP hierarchy is in fact solution of an explicit (yet big) PDE involving only the function $F(N)$  and its derivatives, with no shifts (Theorem~\ref{thm:BKPunshifted} in the appendix). We are not aware of any in-depth study of such ``fixed charge'' BKP equations, which might be worth considering in the future.

\subsection*{Structure of the paper}

In Section~\ref{sec:BKP}, we will recall what we need about the first BKP equations, directing the reader to other sources for the depth of the BKP theory.
In Section~\ref{sec:maps}, we will address the case of maps, taking the time to explain the main ideas and techniques. We will write the Virasoro constraints, and show how to use them to express some derivatives of a specialization of the main BKP tau function as univariate differential polynomials. This will give us "shifted" equations. We will also show how to eliminate the shifts appearing in the BKP equation using instead the first three BKP equations to obtain non-shifted ODEs.
In Section~\ref{sec:bipartite} and Section~\ref{sec:Triangulations}, we will address the cases of bipartite maps and triangulations. The main steps are similar to the case of maps and we will give fewer details than in the previous section.
In Section~\ref{sec:alaCC15}, we apply the technique of elimination of variables of the paper~\cite{CarrellChapuy2015} to obtain slightly different recurrence formulas than in Section~\ref{sec:maps} to count maps.
  
Appendix~\ref{app:Tables} contains tables of the numbers of rooted maps and bipartite maps of genus $g$ with $n$ edges and of rooted triangulations of genus $g$ with $2n$ faces, generated with our recurrences. 
Appendix~\ref{sec:fixedChargeBKP} derives the fixed charge equation for BKP solutions
(Theorem~\ref{thm:BKPunshifted}), which we do not use directly in this paper.

Throughout the paper, the notation $\mathbf{R}[\cdot], \mathbf{R}(\cdot), \mathbf{R}[[\cdot]]$ denote respectively polynomials, rational functions, and formal power series with coefficients in the ring $\mathbf{R}$.

\subsection*{Accompanying Maple worksheet}

A Maple worksheet containing an implementation of the recurrences of this paper, together with automated calculations of the bigger ODEs for the different cases (as well as certain proofs regarding their top coefficients) is available in both Maple and html form in~\cite{us:Maple}.
The worksheet also contains recursive programs obtained from these ODEs, as well tables for small genus and consistency checks against existing formulas of the literature.

\section{A few words on the BKP hierarchy}
\label{sec:BKP}

In this paper, we will use the BKP hierarchy as a black box, and only recall the statements and equations needed for our purposes. We refer the reader to~\cite{KacVandeLeur1998,VandeLeur2001} for the general theory, and to the appendix of our previous paper~\cite{BonzomChapuyDolega2021} for details about the applications to maps and bipartite maps.

The BKP hierarchy is an infinite set of partial differential equations (PDEs) for a sequence of functions $\tau(N)_{N\in\mathbb{Z}}$ depending on ``time parameters'' (formal variables) $p_1,p_2,\dots$. For our combinatorial purposes, it will be convenient to think of the symbol $N$ as a formal variable rather than an integer, and this turns out to be possible under technical conditions, formalized in the notion of ``formal $N$'' BKP tau function in~\cite{BonzomChapuyDolega2021}.

A formal $N$ BKP tau-function is in fact a pair, consisting of a formal power series $\tau(N) \in \mathbb{Q}(N)[[p_1, p_2,\dotsc]]$, together with a normalizing sequence $(\beta_N)_{N\in \mathbb{Z}}$ which is such that
\begin{align}\label{eq:2factorial}
	\frac{\beta_{N-1}\beta_{N+k-1}}{\beta_{N}\beta_{N+k-2}} =R_k(N) \ \ , \ \ 
	\frac{\beta_{N-2}\beta_{N+k}}{\beta_{N}\beta_{N+k-2}} =S_k(N)
\end{align}
for $N\geq 0$, and for respectively every odd positive integer $k$ and every positive integer $k$,
	for some \emph{rational} functions $R_k(N),S_k(N)\in \mathbb{Q}(N)$. These conditions may seem technical but they are crucial to stating the equations of the BKP hierarchy in a formal way as we will do here. In the context of enumeration, the field $\mathbb{Q}$ will often be promoted to a field of rational functions or formal Laurent series involving additional variables, for example $\mathbb{Q}(t)$ so that $R_k(N),S_k(N)\in \mathbb{Q}(t)(N)$.

The typical definition of a (formal or not) BKP tau function makes use of the infinite wedge formalism. It is the image of the orbit of the exponential of an infinite-dimensional Lie algebra, often denoted $b(\infty)$, via the boson-fermion correspondence. We refer the reader to the references mentioned above. For the purposes of this paper, we will admit the PDEs of the BKP hierarchy as a definition:
\begin{definition}
	A pair $(\tau(N),\beta_N)$ as above is a formal-$N$ BKP tau-function
if for $k\in\mathbb{N}, k\geq 1$
the following bilinear identity holds in  $\mathbb{C}(N)[\pp,\qq][[t]]$,
	\begin{multline} \label{BKP}
		\frac{1}{2}\bigl((-1)^k-1\bigr) R_k(N) U(\qq) \tau(N-1)\cdot \tau(N+k-1) 
\\
		+ 
		S_k(N) \sum_j h_j(2\qq) h_{j-k-1}(-\mathbf{\check{D}}) U(\qq) \tau(N-2) \cdot \tau(N+k) \\
		+ \sum_j h_j(-2\qq) h_{j+k-1}(\mathbf{\check{D}}) U(\qq) \tau(N) \cdot \tau(N+k-2) = 0
\end{multline}
	where $\qq = (q_1, q_2, \dots)$ is a vector of formal indeterminates. Here, $h_j$ denotes the complete homogeneous symmetric function of degree $j$, and we define $U(\qq) = e^{\sum_{r\geq 1} q_r D_r}$ and
 $\mathbf{\check{D}} = (k D_k)_{k\geq 1}$, where $D_r$ is the Hirota derivative with respect to $p_r$,
\begin{equation}
D_r f\cdot g = \frac{\partial}{\partial s_r} f(p_r + s_r) g(p_r - s_r)_{|s_r=0}.
\end{equation}
\end{definition}

By extracting coefficients in the variables $q_1,q_2, \dots$ in~\eqref{BKP}, one obtains explicit PDEs for the function $\tau(N)$, which altogether form the BKP hierarchy.
For example, by setting $k=2$ and extracting the coefficient of $q_3$, we obtain the BKP equation~\eqref{BKP1Log} stated in the introduction, where we recall the notation
$$F(N)=\log \tau(N)$$
and where indices indicate  partial derivatives,
$$
f_i := \frac{\partial}{\partial p_i} f.
$$
We will only need two other equations of the hierarchy, namely the following bilinear identities valid in $\mathbb{C}(N)[\pp,\qq][[t]]$:
\begin{multline}
    \label{BKP2Log}
\sminus{}2F_{4,1}(N) \splus{} 2F_{3,2}(N) \splus{} 2F_{2,1}(N)
F_{1^2}(N) \splus{} \tfrac{1}{3} F_{2,1^3}(N) \\
= S_2(N)\frac{\tau(N\!\sminus{}\!2) \tau(N\!\splus{}\!2)}{\tau(N)^2}(F_1(N\!\splus{}\!2)-F_1(N\!\sminus{}\!2)).
\end{multline}
\begin{multline}
  \label{BKP3Log}
\sminus{}6F_{5,1}(N) \splus{} 4F_{4,2}(N) \splus{} 2F_{3^2}(N)\splus{}4F_{3,1}(N) F_{1^2}(N) \splus{} \tfrac{2}{3}F_{3,1^3}(N)
\splus{}4F_{2,1}(N)^2 \\
\splus{} 2F_{2^2}(N) F_{1^2}(N)\splus{} F_{2^2,1^2}(N) \splus \tfrac{1}{3}F_{1^2}(N)^3 \splus \tfrac{1}{6}F_{1^4}(N) F_{1^2}(N) \splus \tfrac{1}{180}F_{1^6}(N) \\
= S_2(N)
\frac{\tau(N\!\sminus{}\!2) \tau(N\!\splus{}\!2)}{\tau(N)^2}\bigl(F_{1^2}(N\!\splus{}\!2)+F_{1^2}(N\!\sminus{}\!2)
+2F_{2}(N\!\splus{}\!2)-2F_{2}(N\!\sminus{}\!2)\\+(F_1(N\!\splus{}\!2)-F_1(N\!\sminus{}\!2))^2\bigr),
\end{multline}
obtained respectively by extracting the coefficient of $q_4$ and $q_5$, again with $k=2$.

We now proceed with map enumeration.

\section{The case of maps}
\label{sec:maps}

\subsection{Generating functions of maps}

For us a \emph{surface} is a non-oriented
two-dimensional real manifold without boundary. 
A surface of Euler characteristic $2-2g$ has \emph{genus} $g$.
A \emph{map} is a graph (with loops and multiple edges allowed) embedded in a surface such that the complement
of the embedding is a disjoint collection of
contractible components, called \emph{faces}. The \emph{genus} of the map is the one of the underlying surface.
A \emph{corner} of a map is a small angular sector around a vertex delimited by two consecutive
edge-sides; the degree of a face/vertex is the number of corners belonging
to it/adjacent to it, respectively. In this paper orientable surfaces do not play a particular role, however in some places we will explicitly use the terminology \emph{non-oriented maps} to emphasize that our surfaces can be orientable or not.

We will be interested
in enumeration of \emph{rooted maps}, i.e.~maps with a distinguished and
oriented corner called the \emph{root corner}. Rooted maps are considered up to homeomorphisms preserving the root corner.
Define the generating function
\begin{equation}
  \label{eq:FunftionF}
  F(t,\pp,u) := \sum_{M} \frac{t^{2e(M)}}{4e(M)} u^{v(M)} \prod_{f \in
    F(M)} p_{\deg(f_i)} ,
\end{equation}
where we sum over all rooted non-oriented maps, where $F(M),E(M),V(M)$
denote the set of faces, edges and vertices of $M$, and
$f(M),e(M),v(M)$ denote their cardinalities.

The following specialization operator plays a crucial role throughout the paper.
\begin{definition}[Specialization $\theta$]
	We let $\theta$ be the operator that specializes all variables $p_i$ to the variable $z$, namely $\theta(p_i) = z$ for every
$i \geq 1$.
\end{definition}
We define the formal power series $\Theta(t,z,u)\in \mathbb{Q}[u,z][[t]]$ 
\begin{equation}
  \label{eq:GeneralMaps}
  \Theta(t,z,u) := \theta F(t,\pp,u) =
  \sum_{M}\frac{t^{2e(M)}}{4e(M)}u^{v(M)}z^{f(M)} = \sum_{n \geq
                  1}\sum_{i,j \geq 1}\frac{H_n^{i,j}}{4n}t^{2n}u^iz^j \\
\end{equation}
which is the bivariate generating function of rooted non-oriented maps
$M$ with variables $t,u,z$ marking respectively twice the number of
edges, the numbers of vertices, and faces, i.e. $H_n^{i,j}$ denotes the number of rooted non-oriented maps with $n$ edges,
$i$ vertices and $j$ faces. It is important to note
that a map with $n$ edges, $i$ vertices, and $j$ faces, has Euler
characteristic $i-n+j=2-2g$, so the genus is implicitely controlled in
this generating function and $\Theta(t,z,u)$ can be rewritten as
\[ \Theta(t,z,u) := \sum_{n \geq
                  1}\sum_{g \geq 0}\frac{H_n^g(u,z)}{4n}t^{2n}, \ \
                \text{where} \ \ H_n^g(u,z) := \sum_{i+j = n+2-2g} H_n^{i,j}u^iz^j.\]
We additionally set $\mathfrak{h}_n^g := H_n^g(1,1)$ for the number of rooted,
non-oriented maps of genus $g$ with $n$ edges and $\mathfrak{u}_n^g :=
H_n^{n+1-2g,1}$ for the number of rooted,
non-oriented maps of genus $g$ with $n$ edges and only one face.

The main goal of this section is to obtain functional equations on the function $\Theta(t,z,u)$, allowing us to compute its coefficients. 
For this, we start from the fact that the ``bigger'' function $F$ has a deep structure inherited from the BKP hierarchy, which was proved by~\cite{VandeLeur2001} using a connection with matrix integrals (see also
\cite[Appendix]{BonzomChapuyDolega2021} for details on the
connection with maps). Here we use the notation $2\pp=(2p_1,2p_2,2p_3,\dots)$.
\begin{proposition}[\cite{VandeLeur2001}]\label{prop:BKPmaps}
Let 
	$\beta_N :=  2^{\lfloor\frac{N}{2}\rfloor} \frac{(2\pi)^{\frac{N}{2}} (2t^2)^{\frac{N(N+1)}{4}}}{N!\ \Gamma(\frac{3}{2})^{N}}   \prod_{j=1}^N \Gamma\bigl(1+\frac{j}{2}\bigr)$. Then the pair $\bigl(\tau(t,2\pp,N) :=
\exp F(t,2\pp,N), \beta_N\bigr)$ 
is a formal $N$ tau function of the BKP
hierarchy. The function $\beta_N$ satisfies \eqref{eq:2factorial} with in particular $S_2(N) = t^4N(N-1)$.
\end{proposition}

Proposition~\ref{prop:BKPmaps} implies that $F(t,2\pp,N)$ satisfies
the BKP equation~\eqref{BKP1Log}. It is tempting to apply the operator
$\theta$ to this equation in order to get information on the function
$\Theta(t,z,u)$, however because partial derivatives with respect to
the $p_i$ do not commute with $\theta$, it is not obvious that such an
approach will succeed. For a sequence of non-negative integers $\lambda=(i_1,\dots,i_k)$, we introduce the quantity
$$
\Fth_\lambda \equiv \Fth_\lambda(t,z,u) := \theta (F_\lambda) = \theta
\left(\prod_{j=1}^k \frac{\partial}{\partial p_{i_j}}  F(t,\pp,u)\right).
$$
The $\Fth_\lambda$ are the quantities naturally appearing when
applying $\theta$ to the BKP equation~\eqref{BKP1Log}.


In order to obtain information on the $\Fth_\lambda$, we use the fact that $\tau(t,\pp,u)$ satisfies the following Virasoro
constraints. 
\begin{proposition}{\cite[Proposition A.1]{BonzomChapuyDolega2021}}
  \label{prop:VirasoroMaps}
	We have $L_i \tau(t,\pp,u) =0$ for $i\geq -1$, where $(L_i)_{i\geq -1}$ are given by
	\begin{align}\label{eq:virmap}
         L_i &= \frac{p_{i+2}^*}{t^2} - \bigg(2\sum_{\substack{a,b\geq1\\a+b=i}}p_a^*p_b^*+ \sum_{a \geq 1} p_a p^*_{a+i} + ((i+1)+2u)p_i^* 
+ \tfrac{\delta_{i,-1}up_1+u(u+1)\delta_{i,0}}{2}\bigg)
\end{align}
and $p_i^* := \frac{i \partial}{\partial p_i}$ for $i>0$ and $p_i^* :=
0$ for $i<1$.
\end{proposition}
	Equation~\eqref{eq:virmap} has a simple combinatorial interpretation corresponding to the deletion of the root edge in a (non necessarily connected) map whose root face has degree $i+2$. It is thus closely related to the Tutte/Lehman-Walsh equations. The term ``Virasoro constraints'' comes from the fact that the operators $L_i$ satisfy the commutation relations of the Virasoro algebra with central charge $c=-2$ \cite{AdlervanMoerbeke2001} -- a fact that we will not use here.  The Virasoro constraints imply the  following proposition:
\begin{proposition}\label{prop:relGlambda}
For $i\geq -1$ and $n_1,n_2,n_3\geq 0$, one has the recurrence relation 
\begin{multline} \label{eq:relGlambda}
\frac{(i+2)\Fth_{i+2,3^{n_3},2^{n_2},1^{n_1}}}{t^2} =
2\sum_{\substack{a+b=i\\a,b\geq 1}} \sum_{\substack{l_i=0\\
    i=1,2,3}}^{n_i} ab\binom{n_1}{l_1} \binom{n_2}{l_2}
\binom{n_3}{l_3} \Fth_{a,3^{l_3},2^{l_2},1^{l_1}} \Fth_{b,3^{n_3-l_3},
  2^{n_2-l_2}, 1^{n_1-l_1}}\\
+2\sum_{\substack{a+b=i\\a,b\geq 1}} ab\Fth_{a,b,3^{n_3},2^{n_2},1^{n_1}}+\sum_{j=1}^3n_j(i+j)\Fth_{i+j,3^{n_3-\delta_{3,j}}, 2^{n_2-\delta_{2,j}}, 1^{n_1-\delta_{1,j}}}+ t\frac{\partial}{\partial t} \Fth_{3^{n_3},2^{n_2},1^{n_1}} \\
-(n_1+2n_2+3n_3) \Fth_{3^{n_3},2^{n_2},1^{n_1}} - z \sum_{a=1}^i
a\Fth_{a,3^{n_3},2^{n_2},1^{n_1}} + \delta_{i \neq -1}(2u+i+1)i \Fth_{i,3^{n_3},2^{n_2},1^{n_1}}\\
+ \Bigl(\delta_{i,-1}(\delta_{n_1,1}+z\delta_{n_1,0}) + (u+1) \delta_{i,0}\delta_{n_1,0}\Bigr)\frac{u}{2}\delta_{n_2,0}\delta_{n_3,0},
\end{multline}
	with the initial condition that $\Fth_{3^{0}2^{0}1^{0}}=\Fth_{\emptyset}=\Theta(t,z,u)$.

	In particular, for any integer vector 
	$\lambda$ of the form $\lambda=[\ell,3^{n_3},2^{n_2},1^{n_1}]$
        with $\ell \leq 9$ and of size $|\lambda| = \ell + n_1 +2n_2 +3n_3$, there exists a polynomial
$P_{\lambda}$ in $\{\frac{\partial^i}{\partial t^i}\Theta(t,z,u): 1
\leq i \leq |\lambda|\}$ with coefficients in $\mathbb{Q}[t,u,z]$
such that
\begin{equation} \label{eq:Flambdadiff}
	\Fth_\lambda = P_\lambda\left(\frac{\partial}{\partial t}\Theta(t,z,u), \dotsc, \frac{\partial^{|\lambda|}}{\partial t^{|\lambda|}}\Theta(t,z,u)\right).
\end{equation}
This polynomial is linear for $\ell\leq 3$ and $|\lambda|\leq 5$, and
quadratic for $4 \leq \ell\leq 6$ and $|\lambda|\leq 6$.
\end{proposition}

\begin{proof}
  The constraints read explicitly
  \begin{multline}
    \label{Virasoro}
(i+2) F_{i+2} =t^2 \Big(2\sum_{a+b=i} ab (F_{a,b} + F_{a}
F_{b})+\sum_{a \geq 1}p_a(i+a)F_{i+a}\Big) \\
+ t^2 (2u+i+1) i F_{i}
+ t^2\frac{u(u+1)}{2} \delta_{i,0} +  t^2 \delta_{i,-1} \frac{p_1u}{2}.
\end{multline}   
We act on both sides of \eqref{Virasoro} with
$\frac{\partial^{n_1+n_2+n_3}}{\partial p_1^{n_1}\partial
  p_2^{n_2}\partial p_3^{n_3}}$ and apply $\theta$. The action on $\sum_{a \geq 1}p_a(i+a)F_{i+a}$ can be re-written using the following equation:
\begin{multline*}
\theta\sum_{a \geq 1}p_a(n\splus{}a)F_{n+a,3^{n_3},2^{n_2},1^{n_1}} 
= z\sum_{a \geq 1} (n\splus{}a)\Fth_{n+a,3^{n_3},2^{n_2},1^{n_1}} \\  
= t\frac{\partial \Fth_{3^{n_3},2^{n_2},1^{n_1}}}{\partial t}  
\sminus(n_1\splus{}2n_2\splus{}3n_3) \Fth_{3^{n_3},2^{n_2},1^{n_1}} 
\sminus z \sum_{a=1}^i
a\Fth_{a,3^{n_3},2^{n_2},1^{n_1}}.
\end{multline*}
It itself comes from the homogeneity relation $\sum_{k\geq1} p_k p_k^* F = t\frac{\partial F}{\partial t}$ by acting with $\frac{\partial^{n_1+n_2+n_3}}{\partial p_1^{n_1}\partial
  p_2^{n_2}\partial p_3^{n_3}}$ and applying $\theta$. This produces \eqref{eq:relGlambda}.

Note that the size of all vectors indexing $\Fth$s appearing in the RHS of
\eqref{eq:relGlambda} are strictly smaller than the size $i+2 + n_1+2n_2+3n_3$ in the LHS.
In order to be able to iterate this equation on all these terms, we need all vectors appearing in the RHS to have at most one part larger than $3$. The only term on the RHS that could have two parts larger than $3$ is
        of the form $\Fth_{a,b,3^{n_3},2^{n_2},1^{n_1}}$. Since $a+b=i$, this does not
	happen unless $i+2> 9$. We thus obtain~\eqref{eq:Flambdadiff}. The last statement is a direct check.
\end{proof}

\begin{remark}
The Virasoro constraints in fact imply a more general result: $\Fth_\lambda$ for any $\lambda$ is a differential polynomial of $\Theta$. This is proved by applying $\prod_{j=1}^k \frac{\partial}{\partial p_{i_j}}$ to \eqref{Virasoro} and performing an induction on $|\lambda|$. We will refrain from writing it since we will not need it in full generality. Instead, the previous proposition is enough to cover the cases we need with explicit formulas, involving vectors of size $|\lambda|\leq 6$.
\end{remark}

We insist on the fact that the recurrence given in Proposition~\ref{prop:relGlambda} can be fully automated to compute the polynomials $P_\lambda$, and this is done in the accompanying Maple worksheet~\cite{us:Maple}.

It is now immediate to see that applying the operator $\theta$ to the BKP equation~\eqref{BKP1Log} produces a functional equation on $\Theta(t,z,u)$. \begin{theorem}
The generating function $\Theta(t,z,u)$ satisfies a 
functional equation of the following form:
\begin{multline*}
\Big(\frac{\partial}{\partial t}\big(\Theta(t,z,u+2)+\Theta(t,z,u-2)-2\Theta(t,z,u)\big)\Big)
  P\left(\frac{\partial}{\partial t}\Theta(t,z,u), \dotsc, \frac{\partial^5}{\partial t^5}\Theta(t,z,u)\right)\\
= Q\left(\frac{\partial}{\partial t}\Theta(t,z,u), \dotsc, \frac{\partial^5}{\partial t^5}\Theta(t,z,u)\right),
  \end{multline*}
where $P$ and $Q$ are quadratic polynomials with coefficients in $\mathbb{Q}[t,u,z]$.
\end{theorem}

\begin{proof}
Consider \eqref{BKP1Log} for $\tau(t,2\pp,u)$. Applying $\theta$ to both sides, taking the 
  derivative with respect to $t$ and substituting the initial equation \eqref{BKP1Log} back into it to eliminate exponentials,
  we obtain
  \begin{multline}
    \label{eq:SmallOdeMaps}\Big(\frac{\partial}{\partial t}\big(\Theta(t,z,u+2)+\Theta(t,z,u-2)-2\Theta(t,z,u)\big)\Big)
  \Big( 4\Fth_{2^2}-4\Fth_{3,1} + \frac{4}{3}(6(\Fth_{1^2})^2 +\Fth_{1^4})\Big)= \\
  \frac{\partial}{\partial
    t}\Big(4\Fth_{2^2}-4\Fth_{3,1} + \frac{4}{3}(6(\Fth_{1^2})^2 +\Fth_{1^4})\Big) - \frac{4}{t} \Big( 4\Fth_{2^2}-4\Fth_{3,1} + \frac{4}{3}(6(\Fth_{1^2})^2 +\Fth_{1^4})\Big)
\end{multline}
thanks to the identity $\frac{\partial}{\partial t}S_2(u)  =
\frac{4}{t}S_2(u)$ implied by $S_2(u) = t^4u(u-1)$.
\cref{prop:relGlambda} immediately concludes the proof.
\end{proof}

The functional equation above can be transformed into a recurrence to compute coefficients. It has the following relatively compact form:
\begin{theorem}[Counting maps by vertices, faces, and genus]  \label{thm:MapsRecurrence}

	The generating polynomial 
	$$H_n^g\equiv H_n^g(u,z) = \sum_{i+j=n+2-2g}H_n^{i,j}u^iz^j$$
	of rooted non-oriented maps of genus $g$ with $n$ edges, with weight $u$ per vertex and $z$ per face, can be computed from the following recurrence formula:
	\begin{multline*} 
	H_n^g+\tfrac{3 u^2 \partial^2}{n(n+1) \partial u^2}H_n^g
=
\tfrac{1}{n(n+1)} \times \Bigg(n\bigg(
	2(2n-1)( (4u+z) H_{n-1}^{g} -2 H_{n-1}^{g-1/2})
	+   4(2 n-3)\big(3uz H_{n-2}^{g}\\
   +  (2 n-1)(n-1)H_{n-2}^{g-1}     \big)
+6\hsum_{g_1=0..g\atop g_1+g_2=g} \sum_{n_1=0..n \atop n_1+n_2=n} 
(2n_1 - 1)(2n_2 - 1)H_{n_2-1}^{g_2}H_{n_1-1}^{g_1}   \bigg)
 \\
-
\hsum_{g_1=0..g\atop g_1+g_2=g} \sum_{n_1=1..n\atop n_1+n_2=n}
	 \hsum_{g_0=0..g_1\atop g_1-g_0\in \mathbb{N}}
	 \sum_{p+j=n_1+2-2g_0}
	 2^{2 (1+g_1-g_0)} {p \choose 2(1+g_1-g_0)}u^{n_1-2g_1-j}z^j H_{n_1}^{p,j}   
\\
 \Big(
	 -\frac{n_2+1}{2}H_{n_2}^{g_2}
+   
		(2n_2-1) \big((4u+z)H_{n_2-1}^{g_2}-2H_{n_2-1}^{g_2-1/2}\big)+
                2(2n_2 - 3)\big((2n_2 - 1)(n_2 - 1)H_{n_2-2}^{g_2-1}\\
                +3uz H_{n_2-2}^{g_2}\big)
+\frac{3}{2}\delta_{g_0\neq g}\delta_{n_1,n}u(\delta_{g_1,g}u- \delta_{g_1,g-1/2})
+\delta_{n_1,n-1}uz(\delta_{g_1,g}(4u + z) - 2\delta_{g_1,g-1/2})\\
+3uz\delta_{n_1,n-2}(\delta_{g_1,g}uz + 2\delta_{g_1,g-1}) 
+ 3\hsum_{g_3=0..g_2\atop g_3+g_4=g_2} \sum_{n_3=0..n_2 \atop n_3+n_4=n_2} 
(2n_3-1)(2n_4-1)H_{n_4-1}^{g_4}H_{n_3-1}^{g_3}   
\Big)\vspace{3cm} \Bigg)
	\end{multline*}
	for $n>2$, with the initial conditions $H_0^g=0$, $H_1^0=uz(u+z)$, $H_2^0=uz(2u^2+5uz+2z^2)$, $H_{1}^{1/2}=uz$, $H_2^{1/2}=5uz(u+z)$, $H_2^1=5uz$, and $H_n^g=0$ if $n<2g$.
\end{theorem}

\begin{proof}
We extract the coefficient of $[t^{2n+4}r^{n+2-2g}]$ in
\eqref{eq:SmallOdeMaps} after substitution $u \to ur, z \to zr$. We obtain from \cref{prop:relGlambda}
	\begin{multline*}
		n\bigg(-(n+1) H_n^g +   
		2(2n-1) \big((4u+z)H_{n-1}^{g}-2H_{n-1}^{g-1/2}\big)+
	 4(2n - 3)\big((2n - 1)(n - 1)H_{n-2}^{g-1}\\
+3uz H_{n-2}^{g}\big) 
+6\hsum_{g_1=0..g\atop g_1+g_2=g} \sum_{n_1=0..n \atop n_1+n_2=n} 
(2n_1-1)(2n_2-1)H_{n_2-1}^{g_2}H_{n_1-1}^{g_1}   \bigg)
 \\
=
\hsum_{g_1=0..g\atop g_1+g_2=g} \sum_{n_1=1..n\atop n_1+n_2=n}
	 \hsum_{g_0=0..g_1\atop g_1-g_0\in \mathbb{N}}
	 \sum_{p+j=n_1+2-2g_0}
	 2^{2 (1+g_1-g_0)} {p \choose 2(1+g_1-g_0)}u^{n_1-2g_1-j}z^j H_{n_1}^{p,j}   
\\
\Big(
	 -\frac{n_2+1}{2}H_{n_2}^{g_2}
+   
		(2n_2-1) \big((4u+z)H_{n_2-1}^{g_2}-2H_{n_2-1}^{g_2-1/2}\big)+
                2(2n_2 - 3)\big((2n_2 - 1)(n_2 - 1)H_{n_2-2}^{g_2-1}\\
                +3uz H_{n_2-2}^{g_2}\big)
+\frac{3}{2}\delta_{n_1,n}u(\delta_{g_1,g}u- \delta_{g_1,g-1/2})
+\delta_{n_1,n-1}uz(\delta_{g_1,g}(4u + z) -
2\delta_{g_1,g-1/2})\\
+3uz\delta_{n_1,n-2}(\delta_{g_1,g}uz 
+ 2\delta_{g_1,g-1}) 
+ 3\hsum_{g_3=0..g_2\atop g_3+g_4=g_2} \sum_{n_3=0..n_2 \atop n_3+n_4=n_2} 
(2n_3-1)(2n_4-1)H_{n_4-1}^{g_4}H_{n_3-1}^{g_3}   
\Big).	
	\end{multline*}
	Here we have extracted, respectively,
        in~\eqref{eq:SmallOdeMaps}  (after substitution $u \to ur, z
        \to zr$), the coefficient of $[t^{2n+4}r^{n+2-2g}]$, of
        $[t^{2n_1-1}r^{n_1-2g_1}]$, and of
        $[t^{2n_2+5}r^{n_2+2-2g_2}]$, in the RHS, in the first factor
        of the LHS, and in the second factor of the LHS. 
Moreover we have used
\begin{align*}
	&[r^m] (f(zr,ur+2)+f(zr,ur-2)-2f(zr,ur))\\ 
&=\sum_{i+j=m} u^i z^j\sum_{p > i} {p \choose i} (2^{p+j-m}+(-2)^{p+j-m} ) [u^pz^j] f(z,u) \\ 
&=\sum_{i+j=m} u^i z^j \sum_{p\geq i+2,\atop p+j-m \in 2\mathbb{N}} {p \choose i} 2^{p+j-m+1}  [u^pz^j] f(z,u)\\
&=\sum_{k\geq m+2\atop k-m \in 2\mathbb{N}}  2^{1+k-m} \sum_{p+j=k}{p \choose m-j}u^{m-j}z^j[u^pz^j] f(z,u).
\end{align*}
In our case $m=n_1-2g_1$, and we parametrized $k$ as $k=n_1+2-2g_0$ (the condition $k\geq m+2$ translates into $g_0 \leq g_1$, and the summand is null when  $g_0<0$).

It now only remains to group the terms of the form $H_n^{i,j}$
with $i+j=n+2-2g$. In the LHS they contribute to the first term
$H_n^g$ and in the RHS they appear as the terms $H_{n_1}^{p,j}$ when $n_1=n, n_2=0,
g_1=g, g_2=0,g_0=g$). Collecting these terms on the RHS gives
\begin{multline*}
  \frac{3}{2}\delta_{n_1,n}\delta_{g_1,g}\delta_{g_0,g}u^2\sum_{p+j=n_1+2-2g_0}
	 2^{2 (1+g_1-g_0)} {p \choose 2(1+g_1-g_0)}u^{n_1-2g_1-j}z^j
         H_{n_1}^{p,j} \\
         = 6\sum_{p+j=n+2-2g} {p \choose 2}u^{p}z^j
         H_{n}^{p,j} = 3\frac{u^2\partial^2}{\partial u^2}H_n^g,
         \end{multline*}
which leads to the main equation of the theorem. The identification of the initial conditions for $n\geq 2$ can be done, either: by hand drawing, or from the OEIS, or from explicit expansions in small genera using the equations of this paper, or from the expansion in Zonal polynomials up to order $n=2$.
\end{proof}

\subsection{Removing the shifts}

We now proceed with the task of obtaining a functional equation on the function $\Theta(t,z,u)$ which does not involve any shift on the variable $u$. We will do this by using the three equations~\eqref{BKP1Log}, \eqref{BKP2Log}, \eqref{BKP3Log} to eliminate the shifts, and apply the operator~$\theta$. This will make terms of the form $\Fth_\lambda$ appear, with larger partitions $\lambda$ than in the previous section, but fortunately they are still in the range covered by Proposition~\ref{prop:relGlambda}.
We have

\begin{theorem}
  \label{theo:ODEMaps}
There exists a polynomial
$P \in \mathbb{Q}[t,u,z][x_1,\dots,x_6]$ of degree $5$ such that
\[ P\left(\frac{\partial}{\partial t}\Theta(t,z,u),\dots,
    \frac{\partial^6}{\partial t^6}\Theta(t,z,u)\right) \equiv 0.\]
An explicit form of $P$ can be obtained by applying 
Proposition~\ref{prop:relGlambda} to the following equation
 \begin{multline}
            \label{eq:ODEMaps}
t^6\Bigl(\frac{\partial}{\partial t}\operatorname{KP1}\Bigr)^2-\Bigl(\operatorname{KP2}\Bigr)^2+\operatorname{KP1}\Bigg(\operatorname{KP3} -\frac{1}{2}\operatorname{KP2}-\\
-\Bigl(t^6 \frac{\partial^2}{\partial t^2}+2t^5
\frac{\partial}{\partial
  t}+2t^6 \frac{\partial^2}{\partial t^2}\Theta+4t^5 \frac{\partial}{\partial
  t}\Theta+ t^4(uz-4)+t^2(3u+1-z)\Bigr)\operatorname{KP1}\Bigg)\equiv 0,
\end{multline}
where
\begin{align*}
 \operatorname{KP1} &= -4\Fth_{3,1} + 4\Fth_{2^2} + \frac{4}{3}(6(\Fth_{1^2})^2 +\Fth_{1^4}),\\
\operatorname{KP2} &= -4\Fth_{4,1} + 4\Fth_{3,2} + \frac{8}{3}(6\Fth_{2,1}\Fth_{1^2}
                     +\Fth_{2,1^3}),\\
\operatorname{KP3}&=-6\Fth_{5,1}+ 4\Fth_{4,2}+2\Fth_{3,3}+
                    \frac{8}{3}(6\Fth_{3,1}\Fth_{1^2}
                    +\Fth_{3,1^3})+4(4(\Fth_{2,1})^2+2\Fth_{2^2}\Fth_{1^2} +
                    \Fth_{2^2,1^2})\\
  &+ \frac{4}{45}(60(\Fth_{1^2})^3 + 30\Fth_{1^4}\Fth_{1^2} + \Fth_{1^6}),
\end{align*}
and $\Theta\equiv \Theta(t,z,u)$. An explicit form of this ODE can be found in the accompanying Maple worksheet~\cite{us:Maple}.
\end{theorem}

\begin{proof}
Denote $\Delta f(u) = f(u+2)-f(u-2)$ and $\nabla f = f(u+2)+f(u-2)$, and
\begin{equation*}
E =S_2(u) e^{\nabla \Theta(t,z,u)-2\Theta(t,z,u)}.
\end{equation*}
Using \eqref{BKP1Log}, \eqref{BKP2Log},
\eqref{BKP3Log} and applying $\theta$ we obtain the
following equations
\begin{equation}
E = \operatorname{KP1}, \qquad E \Delta (\Fth_1) = \operatorname{KP2}, \qquad E(\nabla (\Fth_{1^2})+ \Delta(\Fth_2) + (\Delta \Fth_1)^2) = \operatorname{KP3},
\end{equation}
where $\operatorname{KP1}, \operatorname{KP2},
\operatorname{KP3}$ are given in the statement of theorem.
The difference between $\operatorname{KP1}$, $\operatorname{KP2}$,
$\operatorname{KP3}$ and the LHS in \eqref{BKP1Log}, \eqref{BKP2Log},
\eqref{BKP3Log} is due to the fact that $\tau(N)$ is a formal $N$ tau
function of the BKP hierarchy after rescaling the variables $\pp \to
2\pp$. Using \eqref{eq:relGlambda} and the identity
$S_2(u) = t^4u(u-1)$ we have
\begin{equation*}
  \label{eq:KP1+2Pre}
\Delta (\Fth_1) = t^3 \Delta \frac{\partial}{\partial t}\Theta(t,z,u) + 2t^2 z, \qquad \Delta( \Fth_2) = \frac{t^3}{2} \Delta \frac{\partial}{\partial t}\Theta(t,z,u) + t^2 (2u+1)
\end{equation*}
and
\begin{equation*}
  \label{eq:KP3Pre}
\nabla (\Fth_{1^2}) = \big(t^6 \frac{\partial^2}{\partial t^2}+2t^5
\frac{\partial}{\partial t}\big)\nabla \Theta(t,z,u) +t^4 uz + t^2 u
\end{equation*}
so that the third BKP equation reads
\begin{multline} \label{IntermediateBKP3}
E\Bigg(\left(t^6 \frac{\partial^2}{\partial t^2}+2t^5
\frac{\partial}{\partial t}\right)\nabla \Theta(t,z,u)+\frac{t^3}{2}\Delta
\frac{\partial}{\partial t}\Theta(t,z,u) +t^4 uz + t^2 (3u+1) \\
+ \left(t^3 \Delta \frac{\partial}{\partial t}\Theta(t,z,u)+ 2t^2 z\right)^2\Bigg) = \operatorname{KP3}.
\end{multline}

We now use the first two BKP equations to express $\nabla
\frac{\partial}{\partial t}\Theta(t,z,u)$, $\nabla  \frac{\partial^2}{\partial t^2}\Theta(t,z,u)$ and $\Delta \frac{\partial}{\partial t}\Theta(t,z,u)$ in terms of $\Theta(t,z,u)$ and its $t$-derivatives. Those expressions can then substituted into \eqref{IntermediateBKP3}. Taking the $t$-derivative of the first BKP equation, we have
\begin{equation*}
E\nabla  \frac{\partial}{\partial t}\Theta(t,z,u) = \frac{\partial}{\partial t}\operatorname{KP1}+\Bigl(2 \frac{\partial}{\partial t}\Theta(t,z,u) - \frac{4}{t}\Bigr)\operatorname{KP1},
\end{equation*}
and another derivative gives
\begin{multline*}
E\nabla  \frac{\partial^2}{\partial t^2}\Theta(t,z,u) =
\frac{\partial^2}{\partial t^2}\operatorname{KP1}+\Bigl(2
\frac{\partial}{\partial t}\Theta(t,z,u) - \frac{4}{t}\Bigr)
\frac{\partial}{\partial t}\operatorname{KP1}+\Bigl(2
\frac{\partial^2}{\partial t^2}\Theta(t,z,u) +
\frac{4}{t^2}\Bigr)\operatorname{KP1}\\
-\frac{\partial}{\partial t}\operatorname{KP1}\nabla  \frac{\partial}{\partial t}\Theta(t,z,u)
\end{multline*}
and further
\begin{equation*}
E^2\nabla  \frac{\partial^2}{\partial t^2}\Theta(t,z,u) = \operatorname{KP1}\frac{\partial^2}{\partial t^2}\operatorname{KP1}+\Bigl(2 \frac{\partial^2}{\partial t^2}\Theta(t,z,u) + \frac{4}{t^2}\Bigr)\operatorname{KP1}^2-\Bigl(\frac{\partial}{\partial t}\operatorname{KP1}\Bigr)^2.
\end{equation*}
From the second BKP equation,
\begin{equation*}
Et^3 \Delta \frac{\partial}{\partial t}\Theta(t,z,u) = -2t^2 z\operatorname{KP1} + \operatorname{KP2}. 
\end{equation*}
This gives
\eqref{eq:ODEMaps}. The statement about the form of
polynomial $P$ is a direct consequence of Proposition~\ref{prop:relGlambda}.
\end{proof}

\cref{thm:rec-maps-nonshifted} is an immediate consequence
of \cref{theo:ODEMaps}.
\begin{proof}[Proof of \cref{thm:rec-maps-nonshifted}]
  It is enough to make the change of variables $u \to ur, z\to zr$ in
  \eqref{eq:ODEMaps} and extract the coefficient of
  $[t^{2n}r^{n+2g-2}]$. This substitution allows to track the genus of
  the underlying maps. Extracting the coefficient gives a recursion
  for the bivariate version of $\hh_n^g$ which additionally tracks the number of vertices
  and faces via $u$ and $z$. Specializing $u=z=1$ gives the recursion for
  $\hh_n^g$ of the form \eqref{eq:rec-maps-nonshifted} with $a,b$
  depending on the specific form of ODE given by \eqref{eq:ODEMaps}. A
  direct
  examination of the highest degree terms of this recurrence 
  implemented in~\cite{us:Maple} shows
  that
	\begin{multline*}
        \hh_n^g = \hh_n^{g-1/2}-\sum_{n_1=1..n-1\atop
                  n_1+n_2=n}\sum_{g_1=0..g\atop g_1+g_2=g}\frac{(n_1+1)(n_2+1)}{42(n+1)}\hh_{n_1}^{g_1}\hh_{n_2}^{g_2}\\
         + \sum_{a=1}^{K_1}\hsum_{b=0}^{K_2} \sum_{k=1}^{K_3}
           \sum_{n_1,\dots,n_k \geq 1 \atop n_1+\cdots + n_k =n-a}
           \hsum_{g_1,\dots,g_k \geq 0 \atop  g_1+\dots + g_k =g-b}
	P_{a,b,k}(n_1,\dots,n_k) \hh_{n_1}^{g_1} \hh_{n_2}^{g_2}\dots
           \mathfrak{t}_{n_k}^{g_k},
           	\end{multline*}
which finishes the proof.
  \end{proof}

The coefficient of $z^1$ in $\Theta(t,z,u)$ is the generating function
of maps having only one face, with control on the number of edges and
vertices (equivalently, edges and genus). Extracting the bottom
coefficient in $z$ in \eqref{eq:ODEMaps}, we obtain a \emph{linear}
ODE for this generating function. It is equivalent to Ledoux's recurrence~\eqref{eq:Ledoux} stated in the introduction. 
\begin{corollary}
  The generating function
  \[ \mathfrak{u}(t,u) := [z]\Theta(t,z,u) = \sum_{n \geq 1}\sum_{g \geq 0}\frac{\mathfrak{u}_n^g}{4n}t^{2n}u^{n+1-2g}\]
    of rooted non-oriented maps with only one face satisfies the
    following linear ODE
    \begin{multline}
      \label{eq:ODELedoux}
      \left(32 t^4 \left(u^2-u-5\right)+240 t^6 (2 u-1)+t^2 (10-20 u)+2880 t^8+3\right) \frac{\partial}{\partial t}\mathfrak{u}\\
+t \left(2 t^4 \left(8 u^2-8 u-109\right)+360 t^6 (2 u-1)+t^2 (4-8
  u)+7200 t^8+1\right) \frac{\partial^2}{\partial t^2}\mathfrak{u} \\
+ 6 t^6 \left(20 t^2 (2 u-1)+800 t^4-11\right) \frac{\partial^3}{\partial t^3}\mathfrak{u}+ 5 t^7(-1+2t^2(2u-1)+240t^4) \frac{\partial^4}{\partial t^4}\mathfrak{u} \\+ 120 t^{12} \frac{\partial^5}{\partial t^5}\mathfrak{u} + 4 t^{13}\frac{\partial^6}{\partial t^6}\mathfrak{u}
+240 t^7 u+ 30t^5(2u^2-u)+2t^3( 4u^3-4u^2-11u)-2 t( u^2+ u) \equiv 0.
\end{multline}
\end{corollary}

\section{Recurrences for bipartite maps and triangulations}

\subsection{Non-oriented bipartite maps} \label{sec:bipartite}

Consider the generating function 
\begin{equation}
  G(t,\pp,u,v) := \sum_{M} \frac{t^{e(M)}}{2e(M)} u^{v_{\circ}(M)}v^{v_{\bullet}(M)} \prod_{f \in
    F(M)} p_{\deg(f_i)},
\end{equation}
where we sum over all rooted non-oriented bipartite maps, and
$v_{\circ}(M), v_\bullet(M)$
denote the number of white and black vertices,
respectively. Similarly, as in the case of general maps, the function $G$
inherits a deep structure from the BKP hierarchy. This
result can be derived directly from Van de Leur's work~\cite{VandeLeur2001}, even
though it is not stated explicitly there (see
\cite[Appendix]{BonzomChapuyDolega2021} for additional details on the
connection with maps).

\begin{proposition}[\cite{VandeLeur2001}]\label{prop:BKPBipmaps}
Let $\beta_N = \frac{2^{\lfloor\frac{N}{2}\rfloor}(2t)^{\frac{N(N+\delta)}{2}}}{N!\ \Gamma(3/2)^N} \prod_{j=1}^N \Gamma\bigl(1+\frac{j}{2}\bigr) \Gamma\bigl(\frac{\delta +j}{2}\bigr)$. 
Then the pair $(\tau(t,2\pp,\delta,N) :=
\exp G(t,2\pp,N,N+\delta), \beta_N)$
is a formal $N$ tau function of the BKP
hierarchy. The function $\beta_N$ satisfies \eqref{eq:2factorial} with in particular $S_2(N) = t^4N(N+\delta)(N-1)(N+\delta-1)$.
\end{proposition}

We recall that $\theta(p_i) = z$ for 
$i \geq 1$. Define the power series $\eta(t,z,u,v) \in \mathbb{Q}[u,v,z][[t]]$
\begin{equation}
  \label{eq:BipartiteMaps}
  \eta(t,z,u,v) := \theta G(t,\pp,u,v) =
\sum_{M}\frac{t^{e(M)}}{2e(M)}u^{v_\bullet(M)}v^{v_\circ(M)}z^{f(M)} = \sum_{n,i,j,k\geq 1}\frac{K_n^{i,j,k}}{2n}t^{n}u^iv^jz^k,
\end{equation}
which is the generating function of rooted, non-oriented bipartite
maps $M$. The variables $t,u,v,z$ mark the number of edges, black
vertices, white vertices and faces, respectively so that $K_n^{i,j,k}$
denotes the number of rooted non-oriented bipartite maps with $n$ edges,
$i$ black vertices, $j$ white vertices and $k$ faces (the root vertex
is black by convention). Note that due to the Euler relation we can
rewrite $\eta(t,z,u,v)$ so that it is parametrized by the number of
edges and genus:
\[ \eta(t,z,u,v) := \sum_{n \geq
                  1}\sum_{g \geq 0}\frac{K_n^g(u,v,z)}{2n}t^{n}, \ \
                \text{where} \ \ K_n^g(u,v,z) := \sum_{i+j+k = n+2-2g}
                K_n^{i,j,k}u^iv^jz^k.\]
We additionally set $\mathfrak{k}_n^g := K_n^g(1,1,1)$ for the number of
              rooted non-oriented bipartite maps of genus $g$ with $n$
              edges and $\mathfrak{b}_n^{i,j} :=
K_n^{i,j,1}$ for the number of
              rooted non-oriented bipartite maps of genus $g$ with $n$
              edges, $i$
              black and $j$ white vertices, and only one face.   

In analogy with \cref{prop:relGlambda} we express $\Gth_\lambda$ in
terms of $\frac{\partial^i}{\partial t^i} \eta(t,z,u,v)$, where 
$$
\Gth_\lambda \equiv \Gth_\lambda(t,z,u,v) := \theta (G_\lambda) = \theta
\left(\prod_{j=1}^k \frac{\partial}{\partial p_{i_j}}  G(t,\pp,u,v)\right)
$$
for a sequence of non-negative integers $\lambda=(i_1,\dots,i_k)$.

\begin{proposition}\label{prop:relGlambdaBip}
For $i\geq 0$ and $n_1,n_2,n_3\geq 0$, one has the recurrence relation 
\begin{multline} \label{eq:relGlambdaBip}
\frac{(i+1)\Gth_{i+1,3^{n_3},2^{n_2},1^{n_1}}}{t} \mkern-5mu=\mkern-5mu
2\mkern-8mu\sum_{a+b=i\atop a,b\geq 1} \sum_{l_i=0\atop 
    i=1,2,3}^{n_i} \mkern-8mu ab\binom{n_1}{l_1} \binom{n_2}{l_2}
\binom{n_3}{l_3} \Gth_{a,3^{l_3},2^{l_2},1^{l_1}} \Gth_{b,3^{n_3-l_3},
  2^{n_2-l_2}, 1^{n_1-l_1}}\\
+2\sum_{\substack{a+b=i\\a,b\geq 1}} ab\Gth_{a,b,3^{n_3},2^{n_2},1^{n_1}}+\sum_{j=1}^3n_j(i+j)\Gth_{i+j,3^{n_3-\delta_{3,j}}, 2^{n_2-\delta_{2,j}}, 1^{n_1-\delta_{1,j}}}+ t\frac{\partial}{\partial t} \Gth_{3^{n_3},2^{n_2},1^{n_1}} \\
\sminus{}(n_1\splus{}2n_2\splus{}3n_3) \Gth_{3^{n_3},2^{n_2},1^{n_1}} \sminus{} z\mkern-8mu \sum_{a=1}^i\mkern-5mu
a\Gth_{a,3^{n_3},2^{n_2},1^{n_1}} \splus{} (u+v+i)i \Gth_{i,3^{n_3},2^{n_2},1^{n_1}}\splus{} \tfrac{uv}{2}\delta_{i,0}\delta_{n_1,0}\delta_{n_2,0}\delta_{n_3,0},
\end{multline}
	with the convention that $\Gth_{3^0,2^0,1^0}=\Gth_\emptyset=\eta(t,z,u,v)$.

	In particular, for any partition
	$\lambda$ of the form $\lambda=[\ell,3^{n_3},2^{n_2},1^{n_1}]$ and of size $|\lambda| =\ell+n_1+2n_2+3n_3$, there exists a polynomial
$Q_{\lambda}$ in $\{\frac{\partial^i}{\partial t^i}\eta(t,z,u,v): 1
\leq i \leq |\lambda|\}$ with coefficients in $\mathbb{Q}[t,u,v,z]$
which is linear for $\ell\leq 2$ and $|\lambda|\leq 5$, and
quadratic for $\ell \geq 3, 4 \leq |\lambda|\leq 6$ and satisfies
\begin{align} \label{eq:Glambdadiff}
	\Gth_\lambda = Q_\lambda\left(\frac{\partial}{\partial t}\eta(t,z,u,v), \dotsc, \frac{\partial^{|\lambda|}}{\partial t^{|\lambda|}}\eta(t,z,u,v)\right).
\end{align}
\end{proposition}

The proof is identical to the proof of \cref{prop:relGlambda} (and left to the reader) with the
only difference being in replacing \cref{prop:VirasoroMaps} by its
bipartite counterpart:

\begin{proposition}{\cite[Proposition A.1]{BonzomChapuyDolega2021}}
  \label{prop:VirasoroBipMaps}
	We have $L_i \tau(t,\pp,u,v) =0$ for $i\geq 0$, where $(L_i)_{i\geq 0}$ are given by
\begin{align*}
         L_i &= \frac{p_{i+1}^*}{t}
-\bigg(2\sum_{\substack{a,b\geq1\\a+b=i}}p_a^*p_b^*+ \sum_{a \geq 1} p_a p^*_{a+i} + (i+u+v)p_i^*+\frac{uv\delta_{i,0}}{2}\bigg).
\end{align*}
\end{proposition}

\begin{theorem}[Counting bipartite maps by black/white vertices, faces, and genus]  \label{thm:BipMapsRecurrence}

	The generating polynomial 
	$$K_n^g\equiv K_n^g(u,v,z) = \sum_{i+j+k=n+2-2g}K_n^{i,j,k}u^iv^jz^k$$
	of rooted non-oriented bipartite maps of genus $g$ with $n$
        edges, with weight $u$ per black vertex, $v$ per white vertex and $z$ per face, can be computed from the following recurrence formula:
	\begin{multline*} 
	K_n^g
=
\tfrac{1}{(n+1)}\!
	 \big((2n-1) \big((u + v + z)K_{n-1}^{g}-K_{n-1}^{g-1/2}\big)-
                \psi_n(u,v,z)K_{n-2}^{g}
+(2n - 1)(2n - 3)nK_{n-2}^{g-1}\\
	-6 (n - 1)(u + v - z)K_{n-2}^{g-1/2}
+2\mkern-8mu\hsum_{g_1=0..g\atop g_1+g_2=g} \sum_{n_1=0..n \atop n_1+n_2=n} 
\mkern-8mu(6n_1n_2 - 2(n_1+n_2) + 1) K_{n_2-1}^{g_2}K_{n_1-1}^{g_1}  \Big) 
 \\
-
\mkern-8mu\hsum_{g_1=0..g\atop g_1+g_2=g} \sum_{n_1=1..n-1\atop n_1+n_2=n}
	 \hsum_{g_0=0..g_1\atop g_1-g_0\in \mathbb{N}}\mkern-8mu
	 \tfrac{2^{2 (1+g_1-g_0)}}{(n-2)(n+1)}\mkern-8mu\sum_{p+q+k=\atop n_1+2-2g_0} \sum_{i=0}^{n_1-2g_1-k}{\textstyle {p \choose i}{q \choose n_1-2g_1-k-i}}u^{i}v^{n_1-2g_1-k-i}z^k K_{n_1}^{p,q,k}   
\\
\Big(\sminus{}(n_2+1)K_{n_2}^{g_2} \splus{}
		(2n_2-1) \big((u + v + z)K_{n_2-1}^{g_2}\sminus{}K_{n_2-1}^{g_2-1/2}\big)\splus{}(2n_2 - 1)(2n_2 - 3)n_2
         K_{n_2-2}^{g_2-1}\\
         -6 (n_2 - 1)(u \splus{} v \sminus{}
         z)K_{n_2-2}^{g_2-1/2}+\delta_{n_2,1}\delta_{g_2,0}2uvz+6uv\delta_{n_2,2}(\delta_{g_2,0}uv
         \sminus{} \delta_{g_2,1/2}(u\splus{}v\sminus{}z) \splus{} \delta_{g_2,1}) 
\\ 
-
	 \psi_{n_2}(u,v,z)K_{n_2-2}^{g_2} + 2\hsum_{g_3=0..g\atop g_3+g_4=g_2} \sum_{n_3=0..n_2 \atop n_3+n_4=n_2} 
(6n_3n_4 - 2(n_3+n_4) + 1) K_{n_3-1}^{g_3}K_{n_4-1}^{g_4}  
\Big)
	\end{multline*}
	for $n>2$, with
        \[\psi_n(u,v,z) := (n-2)(u^2 + v^2 + z^2 - 14uv - 2uz -
          2vz)-12uv\]
        and the initial conditions $K_0^g=0$, $K_1^0=uvz$, $K_2^0=uvz(u+v+z)$, $K_{1}^{1/2}=0$, $K_2^{1/2}=uvz$, $K_2^1=0$, and $K_n^g=0$ if $n<2g$.
      \end{theorem}

      \begin{proof}
        The proof is almost identical to the proof of
        \cref{thm:MapsRecurrence}. The only difference is that
        \eqref{eq:SmallOdeMaps} should be replaced by
          \begin{multline*}
    \label{eq:SmallOdeMaps}
\!\Big(\frac{\partial}{\partial t}\big(\eta(t,z,u+2,v+2)+\eta(t,z,u-2,v-2)-2\eta(t,z,u,v)\big)\mkern-3mu\Big)
  \Big( \Gth_{2^2}-\Gth_{3,1} + \frac{1}{3}(6(\Gth_{1^2})^2 +\Gth_{1^4})\Big) \\
  =\frac{\partial}{\partial
    t}\Big(\Gth_{2^2}-\Gth_{3,1} + \frac{1}{3}(6(\Gth_{1^2})^2 +\Gth_{1^4})\Big) - \frac{1}{t} \Big( 4\Gth_{2^2}-4\Gth_{3,1} + \frac{4}{3}(6(\Gth_{1^2})^2 +\Gth_{1^4})\Big).
\end{multline*}
The computational details are left
to the reader.
        \end{proof}

As in the case of maps, it is possible to manipulate the first three BKP equations and obtain an ordinary differential equation on $\eta(t,z,u,v)$. In particular, it does not involve any shifts on the variables $u$ and $v$.

\begin{theorem}
  \label{theo:ODEBipMaps}
There exists a polynomial
$Q \in \mathbb{Q}[t,u,v,z][x_1,\dots,x_6]$ of degree $5$ such that
\[ Q\left(\frac{\partial}{\partial t}\eta(t,z,u,v),\dots,
    \frac{\partial^6}{\partial t^6}\eta(t,z,u,v)\right) \equiv 0.\]
An explicit form of $Q$ can be obtained by applying 
Proposition~\ref{prop:relGlambdaBip} to the following equation
    \begin{multline}
          \label{eq:ODEBipMaps}
t^4\Bigl(\frac{\partial}{\partial t}\operatorname{KP1}\Bigr)^2-\Bigl(\operatorname{KP2}\Bigr)^2+\operatorname{KP1}\Bigg(\operatorname{KP3}-\frac{1}{2}\Bigl(t(u+v+1-z)+1\Bigr)\theta\operatorname{KP2}\\
 -\Bigl(t^4 \frac{\partial^2}{\partial t^2}+4t^3 \frac{\partial}{\partial t}+ 2t^4 \frac{\partial^2}{\partial t^2}\eta +
 8t^3 \frac{\partial}{\partial t}\eta+3uvt^2-(u+v)t\Bigr)\operatorname{KP1}\Bigg)\equiv 0
  \end{multline}
where
\begin{align*}
 \operatorname{KP1} &= -4\Gth_{3,1} + 4\Gth_{2^2} + \frac{4}{3}(6(\Gth_{1^2})^2 +\Gth_{1^4}),\\
\operatorname{KP2} &= -4\Gth_{4,1} + 4\Gth_{3,2} + \frac{8}{3}(6\Gth_{2,1}\Gth_{1^2}
                     +\Gth_{2,1^3}),\\
\operatorname{KP3}&=-6\Gth_{5,1}+ 4\Gth_{4,2}+2\Gth_{3,3}+
                    \frac{8}{3}(6\Gth_{3,1}\Gth_{1^2}
                    +\Gth_{3,1^3})+4(4(\Gth_{2,1})^2+2\Gth_{2^2}\Gth_{1^2} +
                    \Gth_{2^2,1^2})\\
  &+ \frac{4}{45}(60(\Gth_{1^2})^3 + 30\Gth_{1^4}\Gth_{1^2} + \Gth_{1^6}),
\end{align*}
and $\eta \equiv \eta(t,z,u,v)$.
An explicit form of this ODE can be found in the accompanying Maple worksheet~\cite{us:Maple}.
\end{theorem}

The proof is analogous to the proof of \cref{theo:ODEMaps} and left to the reader. We have two immediate corollaries, \cref{thm:rec-Bipmaps-nonshifted} which is analogous to \cref{thm:rec-maps-nonshifted} for bipartite maps, and \cref{thm:BipartiteLedoux} which is a bipartite analogue of Ledoux's recurrence and a non-oriented analogue of Adrianov's.

\begin{theorem}[Counting bipartite maps by edges and genus -- unshifted recurrence]\label{thm:rec-Bipmaps-nonshifted}
The number $\mathfrak{b}_n^g$ of rooted bipartite maps of genus $g$ with $n$ edges, orientable or not, 
is solution of an explicit recurrence relation of the form
	\begin{multline}\label{eq:rec-bipmaps-nonshifted}
		\mathfrak{b}_n^g = \frac{1}{2(n+1)}\bigg(\sum_{n_1=1..n\atop
                   n_1+n_2=n+1}\sum_{g_1=0..g\atop g_1+g_2=g}\delta_{(n_1,g_1)\neq(n,g)}\delta_{(n_2,g_2)\neq(n,g)}
          (n_1+1)(n_2+1)\mathfrak{b}_{n_1}^{g_1}\mathfrak{b}_{n_2}^{g_2}\\
          +\sum_{n_1=1..n-1\atop
                   n_1+n_2=n} (6n_1n_2+5(n_1+n_2)+4)\big(\sum_{g_1=0..g-1\atop g_1+g_2=g-1}\mathfrak{b}_{n_1}^{g_1}\mathfrak{b}_{n_2}^{g_2}-3 \sum_{g_1=0..g\atop g_1+g_2=g}\mathfrak{b}_{n_1}^{g_1}\mathfrak{b}_{n_2}^{g_2}\big)\bigg)\\
          + \sum_{a=1}^{K_1}\hsum_{b=0}^{K_2} \sum_{k=1}^{K_3}
            \sum_{n_1,\dots,n_k \geq 1 \atop n_1+\cdots + n_k =n-a}
            \hsum_{g_1,\dots,g_k \geq 0 \atop  g_1+\dots + g_k =g-b}
	Q_{a,b,k}(n_1,\dots,n_k) \mathfrak{b}_{n_1}^{g_1} \mathfrak{b}_{n_2}^{g_2}\dots \mathfrak{b}_{n_k}^{g_k},
	\end{multline}
	where the $Q_{a,b,k}$ are rational functions and $K_1,K_2,K_3 <\infty$.
      \end{theorem}

\begin{theorem}[A recurrence for non-oriented bipartite one-face maps]\label{thm:BipartiteLedoux}
	The number $\mathfrak{b}_{n}^{i,j}$ of
rooted one-face maps with $n$ edges, $i$ white and $j$ black
vertices, orientable or not, is given by the recursion:
\begin{align}
	(n\splus{}1) \mathfrak{b}_{n}^{i,j} = &\hphantom{+} 
	(4n\sminus{}1)(\mathfrak{b}_{n\sminus{}1}^{i-1,j}\splus{}\mathfrak{b}_{n\sminus{}1}^{i,j-1}\sminus{}\mathfrak{b}_{n\sminus{}1}^{i,j})
	\splus{}(5n^3 \sminus{} 16n^2 \splus{} 13n\sminus{} 1) \mathfrak{b}_{n\sminus{}2}^{i,j}
	\nonumber\\&
\splus{}(2n \sminus{} 3)(4\mathfrak{b}_{n\sminus{}2}^{i-1,j} \splus{} 4\mathfrak{b}_{n\sminus{}2}^{i,j-1} \sminus{}3\mathfrak{b}_{n\sminus{}2}^{i-2,j} \splus{} 3\mathfrak{b}_{n\sminus{}2}^{i,j-2}\sminus{} 2 \mathfrak{b}_{n\sminus{}2}^{i-1,j-1}) 
	\nonumber\\&
\splus{}(10n^3 \sminus{} 68n^2 \splus{} 150n \sminus{}107)(\mathfrak{b}_{n\sminus{}3}^{i,j}\sminus{}\mathfrak{b}_{n\sminus{}3}^{i-1,j}\sminus{}\mathfrak{b}_{n\sminus{}3}^{i,j-1})
	\nonumber\\&
\splus{}(4n	\sminus{}11)(\mathfrak{b}_{n\sminus{}3}^{i-3,j}
\splus{}\mathfrak{b}_{n\sminus{}3}^{i,j-3}\sminus{}\mathfrak{b}_{n\sminus{}3}^{i-2,j-1}
\sminus{}\mathfrak{b}_{n\sminus{}3}^{i-1,j-2} \sminus{}
\mathfrak{b}_{n\sminus{}3}^{i-2,j}\sminus{}\mathfrak{b}_{n\sminus{}3}^{i,j-2} \splus{} 2
	\mathfrak{b}_{n\sminus{}3}^{i-1,j-1})\nonumber\\&
\splus{}(4\sminus{}n)((2n \sminus{} 7)^2(n \sminus{} 2)^2 \mathfrak{b}_{n\sminus{}4}^{i,j}\splus{}(5n^2 \sminus{} 32n \splus{}
53)(\mathfrak{b}_{n\sminus{}4}^{i-2,j}\splus{}\mathfrak{b}_{n\sminus{}4}^{i,j-2}\sminus{}2
	\mathfrak{b}_{n\sminus{}4}^{i-1,j-1})\nonumber\\&
	\splus{}\mathfrak{b}_{n\sminus{}4}^{i-4,j}\splus{}\mathfrak{b}_{n\sminus{}4}^{i,j-4}\sminus{}4 \mathfrak{b}_{n\sminus{}4}^{i-3,j-1}\splus{}4\mathfrak{b}_{n\sminus{}4}^{i-1,j-3}\splus{}6 \mathfrak{b}_{n\sminus{}4}^{i-2,j-2})) 
\end{align}
with the convention that $\mathfrak{b}_{n}^{i,j} =0$ for
$i+j>n+1$, and $\mathfrak{b}_{n}^{i,0}=\mathfrak{b}_{n}^{0,j}=0$ and the initial conditions $\mathfrak{b}_{1}^{1,1} =
\mathfrak{b}_{2}^{2,1} = \mathfrak{b}_{2}^{1,2} =
\mathfrak{b}_{2}^{1,1}= 1$,
$\mathfrak{b}_{3}^{3,1}=\mathfrak{b}_{3}^{1,3}=1$,
$\mathfrak{b}_{3}^{2,2}=\mathfrak{b}_{3}^{2,1}=\mathfrak{b}_{3}^{1,2}=3$,
 $\mathfrak{b}_{3}^{1,1}=4$.
\end{theorem}

\begin{proof}[Proof of \cref{thm:BipartiteLedoux}]
To obtain a linear ODE for the generating function
  \[ \mathfrak{b} \equiv \mathfrak{b}(t,u,v) := [z]\eta(t,z,u,v) = \sum_{n,i,j}\frac{\mathfrak{b}_{n}^{i,j}}{2n}t^nu^iv^j\]
    of rooted non-oriented bipartite maps with only one face, we extract the coefficient of $z^1$ in \eqref{eq:ODEBipMaps}, and
    multiply by $\frac{45}{14t^7uv(u-1)(v-1)}$. This gives
\begin{multline}
(-uv + 2\frac{\partial}{\partial t}\mathfrak{b}) + t(2u^2v + 2uv^2 - 5uv +
7\frac{\partial}{\partial t}\mathfrak{b}(1-u-v) +
\frac{\partial^2}{\partial t^2}\mathfrak{b}) + t^2(-uv((u-v)^2-1) \\
+ (3(3u^2+3v^2+2uv) - 12(u  + v)-29)\frac{\partial}{\partial t}\mathfrak{b} +
4 (1-u - v) \frac{\partial^2}{\partial t^2}\mathfrak{b}) \\
+ t^3(5(-u^3 + u^2v
+ uv^2 - v^3 + (u-v)^2+ 7(u + v - 1)) \frac{\partial}{\partial t}\mathfrak{b} 
+
(2(3u^2+3v^2+2uv) \\
- 8(u  + v)  - 86) \frac{\partial^2}{\partial t^2}\mathfrak{b}) 
+ t^4(((u-v)^4 -
18(u-v)^2 + 81) \frac{\partial}{\partial t}\mathfrak{b} 
- 4(u^3 - u^2v -
uv^2
+v^3\\
- (u-v)^2 -37(u +v  - 1)) \frac{\partial^2}{\partial t^2}\mathfrak{b} -
44 \frac{\partial^3}{\partial t^3}\mathfrak{b}) 
+ t^5(((u - v)^4- 64(u-v)^2+ 719) \frac{\partial^2}{\partial
  t^2}\mathfrak{b} \\
+ 82(u + v  -1) \frac{\partial^3}{\partial t^3}\mathfrak{b}- 5 \frac{\partial^4}{\partial t^4}\mathfrak{b}) 
+ t^6((-38(u-v)^2 + 1078) \frac{\partial^3}{\partial t^3}\mathfrak{b}
\\
+
10(u +v  - 1) \frac{\partial^4}{\partial t^4}\mathfrak{b}) 
+ t^7(-5(u-v)^2+
493) \frac{\partial^4}{\partial t^4}\mathfrak{b}
+ 80t^8 \frac{\partial^5}{\partial t^5}\mathfrak{b} + 4t^9 \frac{\partial^6}{\partial t^6}\mathfrak{b} = 0.
\end{multline}
Extracting the coefficient of $[t^nu^iv^j]$ produces the desired recursion.
  \end{proof}

\subsection{Non-oriented triangulations} \label{sec:Triangulations}

The generating series of triangulations can be obtained from $F(t,\pp,u)$ (given by \eqref{eq:FunftionF}) by applying another specialization instead of $\theta$. Indeed, define the specialization operator $\theta_3$ by $\theta_3(p_i) := z\delta_{3,i}$. This operator enforces that all faces must be of degree 3, and 
\begin{equation*}
  \Xi(t,z,u) := \theta_3 F(t,\pp,u) =
  \sum_{M: \forall f \in F(M)\atop \deg(f)=3}\frac{t^{2e(M)}}{4e(M)}z^{f(M)}u^{v(M)} = \sum_{n \geq 1}
  \sum_{g \geq 0}\frac{\mathfrak{t}_n^g}{12n}t^{6n}z^{2n}u^{n+2-2g},
\end{equation*}
is the generating function of rooted, non-oriented triangulations
$M$. Of course, triangulations satisfy $2e(M)=3f(M)$. By using Euler's relation, one can expand $\Xi(t,z,u)$ by the genus and the
number of edges, and here $\mathfrak{t}_n^g$ denotes the number of
rooted, non-oriented triangulations with $3n$ edges (or equivalently
$2n$ faces) and genus $g$.

Similarly as in the previous sections, we want to express $\Ftr_\lambda$ as a polynomial in $\Xi(t,z,u)$ and its derivatives with respect to $t$, where 
$$
\Ftr_\lambda \equiv \Ftr_\lambda(t,z,u) := \theta_3 (F_\lambda) = \theta_3
\left(\prod_{j=1}^k \frac{\partial}{\partial p_{i_j}}  F(t,\pp,u)\right).
$$
for a sequence of non-negative integers $\lambda=(i_1,\dots,i_k)$.

\begin{proposition}\label{prop:relGlambda3}
For $i\geq -1$ and $n_1,n_2\geq 0$, one has the recurrence relations 
\begin{multline} \label{eq:relGlambda3}
t^2z(i+3)\Ftr_{i+3,2^{n_2},1^{n_1}} =
-2t^2\sum_{\substack{a+b=i\\a,b\geq 1}} \sum_{\substack{l_i=0\\
    i=1,2}}^{n_i} ab\binom{n_1}{l_1} \binom{n_2}{l_2}\Ftr_{a,2^{l_2},1^{l_1}} \Ftr_{b,
  2^{n_2-l_2}, 1^{n_1-l_1}}\\
-2t^2\sum_{\substack{a+b=i\\a,b\geq 1}}
ab\Ftr_{a,b,2^{n_2},1^{n_1}}-t^2\sum_{j=1}^2n_j(i+j)\Ftr_{i+j,2^{n_2-\delta_{2,j}},
  1^{n_1-\delta_{1,j}}}+(i+2)\Ftr_{i+2,2^{n_2},1^{n_1}}\\
- t^2\delta_{i \neq -1}(2u+i+1)i \Ftr_{i,2^{n_2},1^{n_1}}
- t^2\Bigl(\delta_{i,-1}\delta_{n_1,1} + (u+1) \delta_{i,0}\delta_{n_1,0}\Bigr)\frac{u}{2}\delta_{n_2,0},
\end{multline}
\begin{equation}
	\label{eq:relGlambda3-3}
\Ftr_{3,i_1,\dots,i_k} = \frac{t}{3z} \frac{\partial}{\partial
  t}\Ftr_{i_1,\dots,i_k}-\frac{i_1+\cdots+i_k}{3z}\Ftr_{i_1,\dots,i_k}
\end{equation}
\begin{equation}
    \label{eq:relGlambda3-1}
\Ftr_{1^l} = t^5z \frac{\partial}{\partial
  t}\Ftr_{1^{l-1}}+\frac{t^4z(u^2+u)}{2}\delta_{l,1} + \frac{t^2 u}{2}
\delta_{l,2},
\end{equation}
	with the convention that $\Ftr_{3^0,2^0,1^0}=\Xi(t,z,u)$.

	In particular, for any sequence of integers of the form $\lambda=[\ell,3^{n_3},2^{n_2},1^{n_1}]$ and of size $|\lambda|=\ell+n_1+2n_2+3n_3$, there exists a polynomial
$R_{\lambda}$ in $\{\frac{\partial^i}{\partial t^i}\Xi(t,z,u): 1
\leq i \leq |\lambda|\}$ with coefficients in $\mathbb{Q}[t,t^{-1},z,z^{-1},u]$
which is linear for $\ell\leq 4$, and
quadratic for $5 \leq \ell\leq 8$ and satisfies
\begin{align} \label{eq:F3lambdadiff}
	\Ftr_\lambda = R_\lambda\left(\frac{\partial}{\partial t}\Xi(t,z,u), \dotsc, \frac{\partial^{|\lambda|}}{\partial t^{|\lambda|}}\Xi(t,z,u)\right).
\end{align}
\end{proposition}

\begin{proof}
Similarly to the proof of \cref{prop:relGlambda}, we act with $\frac{\partial^{n_1+n_2}}{\partial p_1^{n_1}\partial
  p_2^{n_2}}$ on both sides of \eqref{Virasoro} and apply $\theta_3$ to obtain
\begin{multline}
  \label{eq:pomoc}
(i+2) \Ftr_{i+2,2^{n_2},1^{n_1}} = t^2z (i+3)\Ftr_{i+3,2^{n_2},1^{n_1}} +
  \delta_{i \neq -1}t^2i(i+1+2u)\Ftr_{i,2^{n_2},1^{n_1}} \\
+  2t^2 \sum_{a+b=i} ab\Big(\Ftr_{a,b,2^{n_2},1^{n_1}}+\sum_{l_1=0}^{n_1}\sum_{l_2=0}^{n_2} \binom{n_1}{l_1}\binom{n_2}{l_2} \Ftr_{a, 2^{l_2},1^{l_1}} \Ftr_{b,2^{n_2-l_2},1^{n_1-l_1}}\Big)\\
+ t^2 n_1(i+1) \Ftr_{i+1,2^{n_2},1^{n_1-1}}+ t^2 {n_2}(i+2)\Ftr_{i+2,2^{{n_2}-1},1^{n_1}} 
+ t^2\Bigl(\delta_{i,-1}\delta_{n_1,1} + (u+1) \delta_{i,0}\delta_{n_1,0}\Bigr)\frac{u}{2}\delta_{n_2,0}.
\end{multline}
Equation \eqref{eq:relGlambda3} is obtained from
\eqref{eq:pomoc} by moving $t^2z (i+3)\Ftr_{i+3,2^{n_2},1^{n_1}}$ to the left of the equality and $(i+2) \Ftr_{i+2,2^{n_2},1^{n_1}}$ to the right.

To get Equation \eqref{eq:relGlambda3-3}, one acts on both sides of the homogeneity relation $\sum_{i\geq 1} p_ip_i^* F = t\frac{\partial F}{\partial t}$ with $\frac{\partial^k}{\partial p_{i_1} \dotsm \partial p_{i_k}}$ and applies $\theta_3$.

Furthermore, specializing $i=-1,n_2=0,n_1=l-1$ and then $i=0,n_2=0,n_1=l-1$ in \eqref{eq:pomoc}, we get
\[ \Ftr_{1^l} = t^2z2\Ftr_{2,1^{l-1}}+t^2\frac{u}{2}\delta_{l,2} =
  t^2z\Big(t^2z 3\Ftr_{3,1^{l-1}}+t^2(l-1)
  \Ftr_{1^{l-1}}+t^2\frac{u(u+1)}{2}\delta_{l=1}\Big)
  +t^2\frac{u}{2}\delta_{l,2}.\]
To get \eqref{eq:relGlambda3-1}, one substitutes \eqref{eq:relGlambda3-3} into the above for $i_1, \dotsc, i_k=1$ and $k=l-1$.

Since the sizes of the vectors indexing $\Ftr$s appearing in the RHS of
\eqref{eq:relGlambda3} are strictly smaller than $|\lambda|$, one computes
	$\Fth_\lambda$ recursively for vectors of the form
        $\lambda=[\ell,3^{n_3},2^{n_2},1^{n_1}]$, where $\ell \leq 10$
        (by eliminating all the parts of length $3$ thanks to
        \eqref{eq:relGlambda3-3}, and reducing the sizes of the indexing vectors thanks to \eqref{eq:relGlambda3} and finally using
        the recurrence \eqref{eq:relGlambda3-1} for the parts of the form $\Fth_{1^l}$). The last statement follows by
        induction on $\ell$.
\end{proof}
\begin{remark}
We want to highlight the fact that the above computations are possible
because we are working with the specific model of
triangulations. Replacing the specialization $\theta_3$ by $\theta_l:p_i\mapsto \delta_{i,l}$ with $l
\geq 4$ makes the above technique fail to even compute $F^{\theta_l}_1$.
\end{remark}

\begin{theorem}[Counting triangulations by faces, and genus]
  \label{thm:TriangulationsRecurrence}
	The number $\tttt_n^g$ of rooted non-oriented triangulations of genus $g$
        with $2n$ faces (or, equivalently $3n$ edges) can be computed from the following recurrence formula:
	\begin{multline*} 
	\tttt_n^g
=
\tfrac{2}{2n^2+(3-2g)n+(1-g)(1-2g)} \times \Bigg(
n\bigg(6(3n-1) \tttt_{n-1}^{g}
+   12(3n-4)\big((3n-2)n\tttt_{n-2}^{g-1}\\
+2(\tttt_{n-2}^{g-1/2}+\tttt_{n-2}^{g})\big)
+6\hsum_{g_1=0..g\atop g_1+g_2=g} \sum_{n_1=0..n \atop n_1+n_2=n} 
(3n_1-1)(3n_2-1)\tttt_{n_2-1}^{g_2}\tttt_{n_1-1}^{g_1}   \bigg)
 \\
-\hsum_{g_1=0..g\atop g_1+g_2=g} \sum_{n_1=1..n\atop n_1+n_2=n}
 \left(
	 \hsum_{g_0=0..g_1\atop g_1-g_0\in \mathbb{N}} {\textstyle {n_1+2-2 g_0 \choose n_1-2 g_1}} 
 2^{2 (1+g_1-g_0)}
      \tttt_{n_1}^{g_0}
 \right)\Big(
	 -\frac{n_2+1}{8}\tttt_{n_2}^{g_2}
	 +(3n_2-1){\tttt}_{n_2-1}^{g_2} \\
+2(3n_2-4)\big((3n_2-2)n_2\tttt_{n_2-2}^{g_2-1}
        +2(\tttt_{n_2-2}^{g_2-1/2}+\tttt_{n_2-2}^{g_2})\big)
+ \frac{1}{8}\delta_{g_0\neq g}\delta_{n_1,n}(\delta_{g_1,g}-\delta_{g_1,g-1/2})\\
+\delta_{n_1,n-1}(2\delta_{g_1,g}+2\delta_{g_1,g-1/2}+\delta_{g_1,g-1})+4\delta_{n_1,n-2}(\delta_{g_1,g}+2\delta_{g_1,g-1/2}\\
+9\delta_{g_1,g-1}+8\delta_{g_1,g-3/2})
+\hsum_{g_3=0..g_2\atop g_3+g_4=g_2} \sum_{n_3=0..n_2 \atop n_3+n_4=n_2} 
(3n_3 - 1)(3n_4 - 1)\tttt_{n_3-1}^{g_3}\tttt_{n_4-1}^{g_4}	
\Big)\vspace{3cm} \Bigg)	
	\end{multline*}
	for $n>2$, with the initial conditions 
	$\tttt_0^g=0$, $\tttt_1^0=4$, $\tttt_2^0=32$,
        $\tttt_{1}^{1/2}=9$, $\tttt_2^{1/2}=118$, $\tttt_1^1=7$,
        $\tttt_2^1=202$, $\tttt_2^{3/2}=128$ and $\tttt_n^g=0$ if
        $n<2g-1$.
      \end{theorem}

            \begin{proof}
        As in the case of bipartite maps, the proof is almost identical to the proof of
        \cref{thm:MapsRecurrence} and we leave the details for the
        interested reader. The only difference is that
        \eqref{eq:SmallOdeMaps} should be replaced by
          \begin{multline*}
    \label{eq:SmallOdeMaps}\Big(\frac{\partial}{\partial t}\big(\Xi(t,z,u+2)+\Xi(t,z,u-2)-2\Xi(t,z,u)\big)\Big)
  \Big( 4\Ftr_{2^2}-4\Ftr_{3,1} + \frac{4}{3}(6(\Ftr_{1^2})^2 +\Ftr_{1^4})\Big) \\
  =\frac{\partial}{\partial
    t}\Big(4\Ftr_{2^2}-4\Ftr_{3,1} + \frac{4}{3}(6(\Ftr_{1^2})^2 +\Ftr_{1^4})\Big) - \frac{4}{t} \Big( 4\Ftr_{2^2}-4\Ftr_{3,1} + \frac{4}{3}(6(\Ftr_{1^2})^2 +\Ftr_{1^4})\Big).
\end{multline*}
        \end{proof}

\begin{theorem}
  \label{theo:ODETriangulations}
There exists a polynomial
$R \in \mathbb{Q}[t,u,z][x_1,\dots,x_6]$ of degree $5$ such that
\[ R\left(\frac{\partial}{\partial t}\Xi(t,z,u),\dots,
    \frac{\partial^6}{\partial t^6}\Xi(t,z,u)\right) \equiv 0.\]
An explicit form of $R$ can be obtained by applying 
Proposition~\ref{prop:relGlambda3} to the following equation
  \begin{multline}
    \label{eq:ODETriang}
t^{10}z^2\Bigl(\frac{\partial}{\partial t}\operatorname{KP1}\Bigr)^2-\Bigl(\operatorname{KP2}\Bigr)^2+\operatorname{KP1}\Bigg(\operatorname{KP3}-\frac{1}{2t^2z}\operatorname{KP2}\\
-\Bigl(t^{10}z^2 \frac{\partial^2}{\partial t^2}+5t^9z^2
\frac{\partial}{\partial t} +2t^{10}z^2 \frac{\partial^2}{\partial
  t^2}\Xi + 10t^9z^2 \frac{\partial}{\partial
  t}\Xi+4t^8z^2(u^2+u)+t^2u\Bigr) \operatorname{KP1}\Bigg) \equiv 0
  \end{multline}
where
\begin{align*}
 \operatorname{KP1} &= -4\Ftr_{3,1} + 4\Ftr_{2^2} + \frac{4}{3}(6{\Ftr_{1^2}}^2 +\Ftr_{1^4}),\\
\operatorname{KP2} &= -4\Ftr_{4,1} + 4\Ftr_{3,2} + \frac{8}{3}(6\Ftr_{2,1}\Ftr_{1^2}
+\Ftr_{2,1^3}),\\
\operatorname{KP3}&=-6\Ftr_{5,1}+ 4\Ftr_{4,2}+2\Ftr_{3,3}+
                    \frac{8}{3}(6\Ftr_{3,1}\Ftr_{1^2}
                    +\Ftr_{3,1^3})+4(4{\Ftr_{2,1}}^2+2\Ftr_{2^2}\Ftr_{1^2} +
                    \Ftr_{2^2,1^2})\\
  &+ \frac{4}{45}(60{\Ftr_{1^2}}^3 + 30\Ftr_{1^4}\Ftr_{1^2} + \Ftr_{1^6})
\end{align*}
where $\Xi \equiv \Xi(t,z,u)$.
An explicit form of this ODE can be found in the accompanying Maple worksheet~\cite{us:Maple}.
\end{theorem}

The proof is the same as in the previous cases, so we leave it
as an exercise. As a standard consequence we have.

\begin{theorem}[Counting triangulations by edges and genus -- unshifted recurrence]\label{thm:rec-Triangulations-nonshifted}
The number $\tttt_n^g$ of rooted triangulations of genus $g$ with $3n$ edges, orientable or not, 
is solution of an explicit recurrence relation of the form
	\begin{multline}\label{eq:rec-tri-nonshifted}
          \tttt_n^g = \tttt_n^{g-1/2}-\sum_{n_1=1..n-1\atop
                   n_1+n_2=n}\sum_{g_1=0..g\atop g_1+g_2=g}\frac{(n_1+1)(n_2+1)}{14(n+1)}\tttt_{n_1}^{g_1}\tttt_{n_2}^{g_2}\\
          + \sum_{a=1}^{K_1}\hsum_{b=0}^{K_2} \sum_{k=1}^{K_3}
            \sum_{n_1,\dots,n_k \geq 1 \atop n_1+\cdots + n_k =n-a}
            \hsum_{g_1,\dots,g_k \geq 0 \atop  g_1+\dots + g_k =g-b}
	R_{a,b,k}(n_1,\dots,n_k) \mathfrak{t}_{n_1}^{g_1} \mathfrak{t}_{n_2}^{g_2}\dots \mathfrak{t}_{n_k}^{g_k},
	\end{multline}
	where the $R_{a,b,k}$ are rational functions and $K_1,K_2,K_3 <\infty$.
      \end{theorem}

In analogy to what we did for maps and bipartite maps, it would be natural to study now the case of triangulations with only one vertex (or by duality, cubic one-face maps). However, there exist very explicit and simple formulas in this case, obtained from bijective methods~\cite{BernardiChapuy2011} so we prefer not to go into such calculations here.

\section{Another method in the case of maps}
\label{sec:alaCC15}

In this section, we quickly adress the case of maps treated in Section~\ref{sec:maps} with another method, which actually leads to different recurrence relations. The situation is similar to the orientable case, where the approaches used in~\cite{CarrellChapuy2015} and~\cite{KazarianZograf2015} differ. In Section~\ref{sec:maps} (non-oriented analogue of~\cite{KazarianZograf2015}) we started from the fact that the generating function $F$ of maps is a BKP tau-function, and applied the substitution operator $\theta:p_i\mapsto z$. In this section, we will instead start from the fact that the generating function $G$ of \emph{bipartite} maps is a BKP tau function, and apply the different substitution operator  $\theta_2: p_i \rightarrow \delta_{i,2}$.
We will only treat the equations with shifts, our main motivation being that they are relatively nice looking -- for example we will prove here Theorem~\ref{thm:recurrence2}.

Our starting point is the well-known fact that, from a famous bijection due to Tutte and valid on all surfaces,  the number of rooted maps with $n$ edges on a surface is equal to the number of rooted bipartite quadrangulations on the same surface, with vertices and faces of the map corresponding respectively to black and white vertices of the quadrangulation (see e.g.~\cite{CarrellChapuy2015}). Therefore, the generating function $\Theta(t,z,u)$ of maps defined in Section~\ref{sec:maps} and $G(t,\pp,u,v)$ of bipartite maps defined in Section~\ref{sec:bipartite} satisfy the relation
$$
\theta_2 G(t,\pp,u,z)=\Theta(t,z,u)=\sum_{n \geq
                  1}\sum_{g \geq 0}\frac{H_n^{g}}{4n},
$$
where $H_n^g\in\mathbb{N}[u,z]$ is the generating polynomial of rooted maps of genus $g$ with $n$ edges, with $u$ and $z$ marking respectively vertices and faces.
Let us write $G^{u,z}\equiv G(t,\pp,u,z)$, so that the BKP equation~\eqref{BKP1Log} for the function $G$ can be rewritten

\begin{align}\label{eq:BKPbipartite}
G^{u;z}_{1^4}
+3G^{u;z}_{2^2}
-3 G^{u;z}_{3,1}
+ 6 ( G^{u;z}_{1^2})^2
= C(u,z) \exp\Big( G^{u+2;z+2}-2G^{u;z}+G^{u-2;z-2} \Big),
\end{align}
where the prefactor $C(u,z)$ (which is equal to $S_2(N)$ after the
substitution $u = N, z = N+\delta$), independent of the variables
$(p_i)_{i \geq 1}$, does not play any role in what follows, and where as before the indices indicate derivatives with respect to the $\mathbf{p}$-variables.
By hitting \eqref{eq:BKPbipartite} with the operator $\frac{\partial}{\partial p_1}$, and using~\eqref{eq:BKPbipartite} again  to eliminate the factor $C(u,z)\exp\big(\dots\big)$, we obtain an equation which involves no exponential anymore, and in which the prefactor has disappeared. Namely:
\begin{multline}\label{eq:BKPbipartiteConnected}
G^{u;z}_{1^5}
+3G^{u;z}_{2^2,1}
-3G^{u;z}_{3,1^2}
+ 12G^{u;z}_{1^2} G^{u;z}_{1^3}
=\\
\Big(
	 G^{u+2;z+2}_1
	-2 G^{u;z}_1
	+ G^{u-2;z-2}_1
\Big)
\Big(
G^{u;z}_{1^4}
+3G^{u;z}_{2^2}
-3 G^{u;z}_{3,1}
+ 6 ( G^{u;z}_{1^2} )^2
\Big).
\end{multline}

To extract coefficients in this equation, we will use the following lemma. Here we use an additional variable $r$ which will be convenient to track the genus parameter. In this section, it is convenient to use the convention $H_0^0=uz$.
\begin{lemma}\label{lemma:CC}
	We have, for $(n,g)\in \mathbb{N}\times (\frac{1}{2}\mathbb{N})$, $n\geq 1$, 
\begin{align*}
	[t^{2n+1}r^{n+2-2g}]\theta_2&&\partial_{1}G^{ru;rz} 
	&= \tfrac{1}{2}H_{n}^{g},\\
	[t^{2n}r^{n+2-2g}]\theta_2&&\partial_{2^2}G^{ru;rz} 
	&= \tfrac{n-1}{4} H_n^g \\
	[t^{2n}r^{n+1-2g}]\theta_2&&\partial_{1^2}G^{ru;rz} 
	&= \tfrac{2n-1}{2}H_{n-1}^g,\\
	[t^{2n}r^{n+2-2g}]\theta_2&&\partial_{1^4}G^{ru;rz} 
	&= \tfrac{(2n-1)(2n-2)(2n-3)}{2}H_{n-2}^{g-1},\\
	[t^{2n}r^{n+2-2g}]\theta_2&&\partial_{1,3}G^{ru;rz} 
	&= \tfrac{(2n-1)}{6} \left(H_{n}^{g}-(u+z)H_{n-1}^g - H_{n-1}^{g-1/2}\right),\\
	[t^{2n+1}r^{n+2-2g}]\theta_2&&\partial_{2^2,1}G^{ru;rz} 
	&= \tfrac{n(n-1)}{2}H_{n}^{g},\\
	[t^{2n+1}r^{n+1-2g}]\theta_2&&\partial_{1^3}G^{ru;rz} 
	&= \tfrac{2n(2n-1)}{2}H_{n-1}^{g},\\
	[t^{2n+1}r^{n+2-2g}]\theta_2&&\partial_{1^5}G^{ru;rz} 
	&= \tfrac{2n(2n-1)(2n-2)(2n-3)}{2}H_{n-2}^{g-1},\\
	[t^{2n+1}r^{n+2-2g}]\theta_2&&\partial_{3,1^2}G^{ru;rz} 
	&= \tfrac{n(2n-1)}{3} \left(H_{n}^{g}-(u+z)H_{n-1}^g - H_{n-1}^{g-1/2}\right).
\end{align*}
For $n=0$, the first equality remains valid, while all other quantities in left-hand sides vanish.
\end{lemma}

\begin{proof}
	The lemma can easily be proved with Virasoro constraints in
        the same manner as \cref{prop:relGlambda3} and details are left to the reader. However, a calculation-free proof based on digon contraction and elementary combinatorial map operation is also easily doable. The proof is completely similar to~\cite[Lemma~7]{CarrellChapuy2015}, the only difference is the extra term of genus $g-1/2$ in the two equations involving a hexagonal default (i.e. a $p_3$-derivative). This term comes from the possibility to create a rooted quadrangulation by adding a twisted diagonal inside a digon. This is the only difference between the oriented and non-oriented case, and it adds one term to Equation~(11) in~\cite{CarrellChapuy2015}. Once this difference is taken into account, the proof of~\cite[Lemma~7]{CarrellChapuy2015} can be copied verbatim.
\end{proof}

A direct consequence of what precedes is the recurrence formula stated
as Theorem~\ref{thm:recurrence2} in the introduction. One can also
obtain a version with control on vertices and faces, from which
Theorem~\ref{thm:recurrence2} follows immediately.


\begin{theorem}[Counting maps by vertices, faces, and genus]\label{thm:recurrence2bis}
	The generating polynomial 
	$$H_n^g\equiv H_n^g(u,z) = \sum_{i+j=n+2-2g}H_n^{i,j}u^iz^j$$
	of rooted maps of genus $g$ with $n$ edges, orientable or not, with weight $u$ per vertex and $z$ per face, can be computed from the following recurrence formula:
	\begin{multline} \label{eq:recurrenceCC2}
	H_n^g
=
\tfrac{2}{(n+1)(n-2)} \times \Bigg(
	n(2n-1)( (u+z) H_{n-1}^{g} + H_{n-1}^{g-1/2})
	+   \tfrac{(2 n-3) (2 n-2) (2 n-1) (2 n)}{2} H_{n-2}^{g-1}
\\ 
+12 \hsum_{g_1=0..g\atop g_1+g_2=g} \sum_{n_1=0..n \atop n_1+n_2=n} 
\tfrac{(2 n_2-1)(2 n_1-1)n_1}{2}        H_{n_2-1}^{g_2}          
 H_{n_1-1}^{g_1}   
 \\
-
\hsum_{g_1=0..g\atop g_1+g_2=g} \sum_{n_1=0..n-1\atop n_1+n_2=n}
	 \hsum_{g_0=0..g_1\atop g_1-g_0\in \mathbb{N}}
	 \sum_{p+q=n_1+2-2g_0}
	 2^{2 (1+g_1-g_0)} \phi_{p,q,n_1-2g_1}(u,z) H_{n_1}^{p,q}   
      \\
 \Big(
	 \tfrac{(2 n_2-1) (2 n_2-2) (2 n_2-3)}{2} H_{n_2-2}^{g_2-1}
	 -\delta_{(n_2,g_2)\neq(n,g)}\tfrac{n_2+1}{4} {H}_{n_2}^{g_2} 
	 +\tfrac{2 n_2-1}{2} ( (u+z) H_{n_2-1}^{g_2} + H_{n_2-1}^{g_2-1/2} )\\
+6 
\hsum_{g_3=0..g_2\atop g_3+g_4=g_2}
\sum_{n_3=0..n_2\atop n_3+n_4=n_2}
\tfrac{(2 n_3-1)(2 n_4-1) }{4}H_{n_3-1}^{g_3} H_{n_4-1}^{g_4}
\Big)\vspace{3cm} \Bigg)	
	\end{multline}
	for $n>2$, with the initial conditions 
	$H_0^0=uz$, $H_1^0=uz(u+z)$, $H_2^0=uz(2u^2+5uz+2z^2)$, $H_{1}^{1/2}=uz$, $H_2^{1/2}=5uz(u+z)$, $H_2^1=5uz$, and $H_n^g=0$ if $n<2g$, and where 
	$$\phi_{p,q,m}(u,z):= \sum_{i+j=m \atop p\geq i, q\geq j} {\textstyle {p \choose i} {q \choose j} u^i z^j}
.$$
\end{theorem}


\begin{proof}

	We substitute $u \to ur, z \to zr$ in
        \eqref{eq:BKPbipartiteConnected} and we extract the
        coefficient of $[t^{2n+1}r^{n+2-2g}]$ after applying
        $\theta_2$. We obtain from Lemma~\ref{lemma:CC}

	\begin{multline*}
		-\tfrac{n(n+1)}{2}H_n^g +   
		n(2n-1)( (u+z) H_{n-1}^{g} + H_{n-1}^{g-1/2})
	+   \tfrac{(2 n-3) (2 n-2) (2 n-1) (2 n)}{2} H_{n-2}^{g-1}
\\ 
+12 \hsum_{g_1=0..g\atop g_1+g_2=g} \sum_{n_1=0..n \atop n_1+n_2=n} 
\tfrac{(2 n_2-1)(2 n_1-1)n_1}{2}        H_{n_2-1}^{g_2}          
 H_{n_1-1}^{g_1}   
 \\
=
\hsum_{g_1=0..g\atop g_1+g_2=g} \sum_{n_1=0..n-1\atop n_1+n_2=n}
 \left(
	 \hsum_{g_0=0..g_1\atop g_1-g_0\in \mathbb{N}} \sum_{p+q=n_1+2-2g_0}
	 2^{2 (1+g_1-g_0)} \phi_{p,q,n_1-2g_1}(u,z) H_{n_1}^{p,q}   
 \right)\\
 \Big(
	 \tfrac{(2 n_2-1) (2 n_2-2) (2 n_2-3)}{2} H_{n_2-2}^{g_2-1}
	 +\tfrac{-n_2-1}{4} {H}_{n_2}^{g_2} 
 +\tfrac{2 n_2-1}{2} ( (u+z) H_{n_2-1}^{g_2} + H_{n_2-1}^{g_2-1/2} )\\
+6 
\hsum_{g_3=0..g_2\atop g_3+g_4=g_2}
\sum_{n_3=0..n_2\atop n_3+n_4=n_2}
\tfrac{(2 n_3-1)(2 n_4-1) }{4}H_{n_3-1}^{g_3} H_{n_4-1}^{g_4}
\Big).	
	\end{multline*}
	Here we have extracted, respectively,
        in~\eqref{eq:BKPbipartiteConnected} (after applying $\theta_2$,
        and substitution $u \to ur, z \to zr$) the coefficient of $[t^{2n+1}r^{n+2-2g}]$, of $[t^{2n_1+1}r^{n_1-2g_1}]$, and of $[t^{2n_2}r^{n_2+2-2g_2}]$, in the LHS, in the first factor of the RHS, and in the second factor of the RHS. The summation over $n_1$ in the RHS stops at $(n-1)$ since from the last sentence of Lemma~\ref{lemma:CC}, the term $n_2=0$ does not contribute.
Moreover we have used
\begin{align*}
	&[r^m] (f(ru+2, rz+2)+f(ru-2, rz-2)-2f(ru,rz))\\ 
&=\sum_{i+j=m} u^i z^j\sum_{p\geq i, q\geq j\atop p+q\neq m} {p \choose i}{q \choose j}  (2^{p+q-m}+(-2)^{p+q-m} ) [u^pz^q] f(u,z) \\ 
&=\sum_{i+j=m} u^i z^j \sum_{p\geq i, q\geq j\atop p+q\geq m+2, p+q-m \in 2\mathbb{N}} {p \choose i}{q \choose j}  2^{1+p+q-m}  [u^pz^q] f(u,z)\\
&=\sum_{k\geq m+2\atop k-m \in 2\mathbb{N}}  2^{1+k-m} \sum_{p+q=k}\phi_{p,q,m}(u,z) [u^pz^q] f(u,z) 
,
\end{align*}
which we use with the parametrization $m=n_1-2g_1$, $k=n_1+2-2g_0$
(the condition $k\geq m+2$ translates into $g_0 \leq g_1$, and the summand is null when  $g_0<0$).

It now only remains to group the two terms of the form $H_n^g$ (namely: the first term of the LHS and the term $H_{n_2}^{g_2}$ in the RHS when $n_1=0, n_2=n, g_1=0, g_2=g$). They appear in the difference LHS-RHS with coefficient
$-\frac{n(n+1)}{2} +2^2\times  \frac{n+1}{4} = -\frac{(n+1)(n-2)}{2}$,
which leads to the main equation of the theorem after dividing by this factor. The identification of the initial conditions for $n\geq 2$ can be done in many ways including: by hand drawing, or from the OEIS, or from explicit expansions in small genera using the Virasoro constraints, or from the expansion in Zonal polynomials up to order $n=2$.
\end{proof}

\begin{proof}[Proof of Theorem~\ref{thm:recurrence2} stated in the introduction]
Note that $\hh_n^g = H_n^g(1,1)$ and
  \[ \phi_{p,q,n_1-2g_1}(1,1) = \sum_{i+j=m \atop p\geq i, q\geq j}
    {\textstyle {p \choose i} {q \choose j}} = {\textstyle {p+q
        \choose n_1-2g_1}}, \]
  therefore \eqref{eq:recurrenceCC} is a specialization of
  \eqref{eq:recurrenceCC2} at $u=z=1$.
\end{proof}

\vfill

\appendix

\section{Some tables} \label{app:Tables}

We provide here some tables computed with our recurrences, see also~\cite{us:Maple}.

\subsection{Rooted maps of genus \texorpdfstring{$g$}{g} with \texorpdfstring{$n$}{n} edges (orientable or not)}

\begin{center}
\begin{tiny}
\begin{tabular}{c|ccccc}
\midrule
$ n\backslash g$&\phantom{0}$0$\phantom{0}&\phantom{0}$1/2$\phantom{0}&\phantom{0}$1$\phantom{0}&\phantom{0}$3/2$\phantom{0}&\phantom{0}$2$\phantom{0}\\ 
  \midrule
  1 &                       2 &                      1 &0 &0&     0\\
 2 &                       9 &                       10 &5 &0&     0\\
 3 &                       54 &                       98 &104 &41 & 0\\
 4 &                      378 &                       983 &1647 &1380 & 509\\
 5 &                     2916 &                      10062 &23560 &31225  & 24286\\
 6 &                    24057 &                   105024 &320198 &592824& 724866\\
 7 &                   208494 &                  1112757 &4222792 &10185056& 17312568\\
 8 &                 1876446 &                11934910 &54617267 &164037704& 361811054\\
 9 &                17399772 &               129307100 &696972524 &2525186319& 6912864180\\
10 &               165297834 &             1412855500 &8807574390 &37596421940& 123814835628\\
11 &              1602117468 &            15548498902 &110483092984 &545585129474& 2111880200672\\
12 &             15792300756 &           172168201088 &1377998069826&7758174844664& 34669329147582\\
13 &            157923007560 &         1916619748084 &17108920039328&108518545261360& 551879941676492\\
14 &           1598970451545 &         21436209373224 &211636362018548&1497384373878512& 8565305839025180\\
15 &          16365932856990 &         240741065193282 &2609949110616064&20426386710028260&130146976774282440\\
16 &          169114639522230 &         2713584138389838 &32104324480419131&275940187259609296&1942255149093281772\\
  \midrule
\end{tabular}
%

\begin{tabular}{c|cccc}
\midrule
$ n\backslash g $ &\phantom{00000}$5/2$\phantom{00000}&\phantom{0000}$3$\phantom{0000}&\phantom{0000000000}$7/2$\phantom{0000000000}&\phantom{0000000}$4$\phantom{0000000}\\ 
\midrule
 5 &                    8229 &                      0 &                       0 &                       0\\
 6 &                   516958 &                   166377 &                    0 &                       0\\
 7 &                  19381145 &                  13093972 &                   4016613 &                       0\\
 8 &                 562395292 &                 595145086 &                 382630152 &                    113044185\\
 9 &                13929564070 &                20431929240 &                20549348578 &                  12704958810\\
10 &               309411522140 &              587509756150 &              818177659640 &                790343495467\\
11 &              6344707786945 &             14923379377192 &             26881028060634 &              35918779737610\\
12 &             122357481545872 &            345651571125768 &            770725841809552 &             1330964564940140\\
13 &            2247532739398856 &           7452363840633244 &          19946409152977346 &           42611002435124552\\
14 &           39681114425793904 &          151717486205709730 &         476412224477845444 &          1220973091185233106\\
15 &          677939355268197412 &         2946794762696249280 &        10665684328125155376 &         32054128913697072040\\
16 &         11265765391845733784 &         55029552840385680100 &      226357454725004343024&       783804517126931727890\\
  \midrule
\end{tabular}
\end{tiny}
\end{center}

\subsection{Rooted bipartite maps of genus \texorpdfstring{$g$}{g} with \texorpdfstring{$n$}{n} edges (orientable or not)}

\begin{center}
\begin{tiny}
\begin{tabular}{c|ccccc}
\midrule
$ n\backslash g $ &\phantom{0}$0$\phantom{0}&\phantom{0}$1/2$\phantom{0}&\phantom{0}$1$\phantom{0}&\phantom{0}$3/2$\phantom{0}&\phantom{0}$2$\phantom{0}\\ 
  \midrule
  1 &                       1 &                      0 &0 &0&     0\\
 2 &                       3 &                       1 &0 &0&     0\\
 3 &                       12 &                       9 &4 &0 & 0\\
 4 &                      56 &                       69 &63 &20& 0\\
 5 &                    288 &                     510 &720 &480& 148\\
 6 &                   1584 &                  3738 &7254 &7584& 4860\\
 7 &                  9152 &                27405 &68460 &99372& 99036\\
 8 &                54912 &               201569 &621315 &1169640& 1607432\\
 9 &               339456 &              1488762 &5496208 &12841632& 22759560\\
10 &              2149888 &            11043318 &47759130 &134278720& 293971176\\
11 &             13891584 &           82257890 &409620156 &1354371348& 3553592152\\
12 &            91287552 &          615092178 &3478672642&13287239184& 40855164228\\
13 &           608583680 &        4615882908 &29315742924&127526774024& 451592018748\\
14 &           4107939840 &        34752865332 &245539064736&1202371430148& 4836001359644\\
15 &         28030648320 &        262437282621 &2046309441924&11170818315900&50454786158100\\
16 &         193100021760 &        1987229885913 &16983591315267&102508926612240&515031678182160\\
  \midrule
\end{tabular}

\begin{tabular}{c|cccc}
\midrule
$ n\backslash g $ &\phantom{00000}$5/2$\phantom{00000}&\phantom{0000}$3$\phantom{0000}&\phantom{0000000000}$7/2$\phantom{0000000000}&\phantom{0000000}$4$\phantom{0000000}\\ 
\midrule
 6 &                  1348 &                   0 &                     0 &                       0\\
 7 &                 57204 &                 15104 &                   0 &                       0\\
 8 &                1445760 &                793260 &                198144 &                    0\\
 9 &               28251720 &               24092916 &               12500640 &                  2998656\\
10 &              470885712 &             553335140 &             446044020 &                222034464\\
11 &             7034561160 &            10652501508 &            11827897444 &              9139492032\\
12 &            96964428080 &           181251943620 &           259263273912 &             275741173612\\
13 &           1256403317832 &          2812951666460 &         4965451637328 &           6799083573828\\
14 &          15499423803780 &         40643437847436 &        85911625991020 &          145094953853052\\
15 &         183709516250796 &        554529301430940 &       1372607347932900 &         2774708761422460\\
16 &        2106284848285632 &        7218066635434760 &      20563312515574176&       48658560979911312\\
  \midrule
\end{tabular}
\end{tiny}
\end{center}

\subsection{Rooted triangulations of genus \texorpdfstring{$g$}{g} with \texorpdfstring{$2n$}{2n} faces (orientable or not)}

\begin{center}
\begin{tiny}
\begin{tabular}{c|ccc}
\midrule
$ n\backslash g $ &\phantom{0}$0$\phantom{0}&\phantom{0}$1/2$\phantom{0}&\phantom{0}$1$\phantom{0}\\ 
  \midrule
  1 &                       4 &                      9 &      7     \\
 2 &                       32 &                       118 & 202     \\
 3 &                       336 &                       1773 &4900   \\
 4 &                      4096 &                      28650 &112046 \\
 5 &                    54912 &                     484578 &2490132 \\
 6 &                   786432&                  8457708 &54442636   \\
 7 &                  11824384 &              151054173 &1177912344 \\
 8 &                184549376 &             2745685954 &25302706734 \\
 9 &               2966845440 &           50606020854 &540709469284 \\
10 &              48855252992 &        943283037684 &11509659737732 \\
11 &              820675092480 &    17746990547634 &244254583041960 \\
12 &            14018773254144 &   336517405188900 &5170993925895980\\
13 &          242919827374080 & 6423775409047716 &109258058984867592\\
14 &       4261707069259776 &123332141503711704 &2304778527410416728\\
15 &    75576645116559360 &2379824766494404317 &48552885599587471920\\
  \midrule
\end{tabular}

\begin{tabular}{c|ccc}
\midrule
$ n\backslash g $ &\phantom{0}$3/2$\phantom{0}&\phantom{0}$2$\phantom{0}&\phantom{00000}$5/2$\phantom{00000}\\ 
\midrule
 1 &   0& 0& 0\\
 2 &  128 &     0&  0  \\
 3 &    6786 & 3885& 0    \\
 4 &249416& 309792 & 163840      \\
 5 &7820190& 15536592 & 17742726     \\
 6 &224154528& 626073960& 1140086560   \\
 7 &6064485588& 22147258392& 56574101430   \\
 8 &157592065776& 718135826112& 2394618429216   \\
9 &3975252852294& 21875815507824& 90903502798380   \\
10 &98013064376240& 635740513124184& 3186926652389376 \\
11 &2373323509105164& 17808561973715832& 105134232237568182 \\
12 &56632532943141168& 484348105828421472& 3305475583204245376 \\
13 &1335091307453227116& 12857728996745420112& 99951709676667034212 \\
14 &31155184166556067968&334487003255090327376& 2926388895694344300864 \\
15 &720738499764872647080&8553392225715199201200& 83383518174303020028732 \\
  \midrule
\end{tabular}

\begin{tabular}{c|ccc}
\midrule
$ n\backslash g $ &\phantom{0000}$3$\phantom{0000}&\phantom{0000000000}$7/2$\phantom{0000000000}&\phantom{0000000}$4$\phantom{0000000}\\ 
\midrule
 4 &                                  0 &                       0 &                       0\\
 5 &                             8878870 &                      0 &                       0\\
 6 &                 1227058016 &                  587202560 &                       0\\
 7 &                96836144376 &                99359372628 &                    45877917085\\
 8 &               5738654714432 &               9344829276160 &                  9227542480640\\
9 &               283959455776728 &             646430229699516 &                1011244742721480\\
10 &             12391917590699520 &            36729466978572288 &              80222081136864896\\
11 &            492702239182522512 &           1816211696054002632 &             5159782135287908304\\
12 &           18229054925434379424 &         80930038930104447744 &           285723552389864612352\\
13 &           636795227835309684024 &        3324906134317505727756&          14128927461188199914592\\
14 &         21225085309259820837824 &       127965696661596592413184 &         639196545524077326637824\\
15 &        680321493460375920656880 &       4667484955217376877322616&       26909217174327495052218480\\
  \midrule
\end{tabular}
\end{tiny}
\end{center}

\section{A fixed-charge BKP equation}
\label{sec:fixedChargeBKP}

\begin{theorem}\label{thm:BKPunshifted}
	Let $\tau(N)$ be a  BKP tau function. Then for $k\in\mathbb{N}, k\geq 1$
the following identity holds in
$\mathbb{C}(N)[\pp,\qq][[t]]$:
\begin{multline}
  \label{eq:FixedCharge}
  2 F_{1^3}\operatorname{KP1}^3 =  (\operatorname{KP3}_1-2 \operatorname{KP2}_2)\operatorname{KP1}^2 -
(\operatorname{KP3}-3\operatorname{KP1}_{1^2}) \operatorname{KP1} \operatorname{KP1}_1  \\+2 (\operatorname{KP1}_2-\operatorname{KP2}_1)\operatorname{KP1}\operatorname{KP2}+ 2 \operatorname{KP2}^2 \operatorname{KP1}_1 
-2 \operatorname{KP1}_1^3 - \operatorname{KP1}^2 \operatorname{KP1}_{1^3},
\end{multline}
where
\begin{align*}
 \operatorname{KP1} &= -F_{3,1} + F_{2^2} + \frac{1}{2}F_{1^2}^2 +\frac{1}{12}F_{1^4},\\
\operatorname{KP2} &= -2F_{4,1} + 2F_{3,2} + 2F_{2,1}F_{1^2}
                     +\frac{1}{3} F_{2,1^3},\\
\operatorname{KP3}&=-6F_{5,1}+ 4F_{4,2}+2F_{3^2}+
                    4F_{3,1}F_{1^2}
                    +\frac{2}{3}F_{3,1^3}+4F_{2,1}^2+2F_{2^2}F_{1^2} +
                    F_{2^2,1^2}\\
                    &+ \frac{1}{3}F_{1^2}^3 + \frac{1}{6} F_{1^4}F_{1^2} + \frac{1}{180} F_{1^6},
\end{align*}
	with $F\equiv F(N)  = \log \tau(N)$.
\end{theorem}

\begin{proof}
Denote $\Delta f(N) = f(N+2)-f(N-2)$ and $\nabla f = f(N+2)+f(N-2)$, and
\begin{equation}
E =S_2(N) e^{\nabla F(N)-2F(N)}.
\end{equation}
Using \eqref{BKP1Log}, \eqref{BKP2Log},
\eqref{BKP3Log} we obtain the
following equations
\begin{equation}
E = \operatorname{KP1}, \qquad E \Delta (F_1) = \operatorname{KP2}, \qquad E(\nabla(F_{1^2})+ 2\Delta(F_2) + (\Delta F_1)^2) = \operatorname{KP3},
\end{equation}
where $\operatorname{KP1}, \operatorname{KP2},
\operatorname{KP3}$ are given in the statement and they  are the LHS in \eqref{BKP1Log}, \eqref{BKP2Log},
\eqref{BKP3Log}.
We need to differentiate the 3rd equation w.r.t. $p_1$,
\begin{equation}
\Delta 2F_{2,1} + \nabla F_{1^3} + 2\Delta F_1 \Delta F_{1^2} = \Bigl(\frac{\operatorname{KP3}}{\operatorname{KP1}}\Bigr)_1
\end{equation}
and rewrite the LHS using derivatives of the first 2 BKP equations,
\begin{multline}
2\Delta F_{2,1} + \nabla F_{1^3} + 2\Delta F_1 \Delta F_{1^2} = 2\Bigl(\frac{\operatorname{KP2}}{\operatorname{KP1}}\Bigr)_2 + 2 F_{1^3} + \Bigl(\frac{\operatorname{KP1}_1}{\operatorname{KP1}}\Bigr)_{1^2} + 2 \Bigl(\frac{\operatorname{KP2}}{\operatorname{KP1}}\Bigr) \Bigl(\frac{\operatorname{KP2}}{\operatorname{KP1}}\Bigr)_1\\
= 2 F_{1^3} + 
\frac{1}{\operatorname{KP1}^3}\Bigl(-2\operatorname{KP1}\operatorname{KP2}\operatorname{KP1}_2 + 2\operatorname{KP1}^2 \operatorname{KP2}_2 - 2 \operatorname{KP2}^2 \operatorname{KP1}_1 \\
+2 \operatorname{KP1}_1^3 + 2 \operatorname{KP1}\operatorname{KP2} \operatorname{KP2}_1 - 3 \operatorname{KP1} \operatorname{KP1}_1 \operatorname{KP1}_{1^2} + \operatorname{KP1}^2 \operatorname{KP1}_{1^3}\Bigr).
\end{multline}
We finally equate the RHS of the above two equations and multiply by
$\operatorname{KP1}^3$ to obtain \eqref{eq:FixedCharge}.
\end{proof}

\bibliographystyle{amsalpha}
\bibliography{biblio2015}

\def\cprime{$'$}
\providecommand{\bysame}{\leavevmode\hbox to3em{\hrulefill}\thinspace}
\providecommand{\MR}{\relax\ifhmode\unskip\space\fi MR }
\providecommand{\MRhref}[2]{%
  \href{http://www.ams.org/mathscinet-getitem?mr=#1}{#2}
}
\providecommand{\href}[2]{#2}
\begin{thebibliography}{BCEGF21}

\bibitem[ACEH20]{AlexandrovChapuyEynardHarnad2020}
A.~Alexandrov, G.~Chapuy, B.~Eynard, and J.~Harnad, \emph{Weighted {H}urwitz
  numbers and topological recursion}, Comm. Math. Phys. \textbf{375} (2020),
  no.~1, 237--305. \MR{4082183}

\bibitem[Adr97]{Adrianov1997}
N.~M. Adrianov, \emph{An analog of the harer--zagier formula for unicellular
  bicolored maps}, Funktsional. Anal. i Prilozhen. \textbf{31} (1997), no.~3,
  1--9.

\bibitem[AL19]{AlbenqueLepoutre2019}
M.~Albenque and M.~Lepoutre, \emph{Blossoming bijection for higher-genus maps,
  and bivariate rationality}, Preprint arXiv:2007.07692, 2019.

\bibitem[AvM01]{AdlervanMoerbeke2001}
M.~Adler and P.~van Moerbeke, \emph{Hermitian, symmetric and symplectic random
  ensembles: {PDE}s for the distribution of the spectrum}, Ann. of Math. (2)
  \textbf{153} (2001), no.~1, 149--189. \MR{1826412}

\bibitem[BBM11]{BernardiBousquetMelou2011}
O.~Bernardi and M.~Bousquet-M{\'{e}}lou, \emph{Counting colored planar maps:
  Algebraicity results}, Journal of Combinatorial Theory, Series B \textbf{101}
  (2011), no.~5, 315--377.

\bibitem[BBM17]{BernardiBousquetMelou2017}
\bysame, \emph{Counting coloured planar maps: Differential equations},
  Communications in Mathematical Physics \textbf{354} (2017), no.~1, 31--84.

\bibitem[BC86]{BenderCanfield1986}
E.~A. Bender and E.~R. Canfield, \emph{{The asymptotic number of rooted maps on
  a surface}}, J. Combin. Theory Ser. A \textbf{43} (1986), no.~2, 244--257.
  \MR{867650 (88a:05080)}

\bibitem[BC91]{BenderCanfield1991}
\bysame, \emph{The number of rooted maps on an orientable surface}, J. Combin.
  Theory Ser. B \textbf{53} (1991), no.~2, 293--299. \MR{1129556}

\bibitem[BC94]{BenderCanfield1994}
\bysame, \emph{The number of degree-restricted rooted maps on the sphere}, SIAM
  J. Discrete Math. \textbf{7} (1994), no.~1, 9--15. \MR{1259005}

\bibitem[BC11]{BernardiChapuy2011}
O.~Bernardi and G.~Chapuy, \emph{{Counting unicellular maps on non-orientable
  surfaces}}, Adv. in Appl. Math. \textbf{47} (2011), no.~2, 259--275.
  \MR{2803802 (2012i:05012)}

\bibitem[BCD21a]{BonzomChapuyDolega2021}
V.~Bonzom, G.~Chapuy, and M.~Do{\l}{\k e}ga, \emph{$b$-monotone {H}urwitz
  numbers: {V}irasoro constraints, {BKP} hierarchy, and {$O(N)$-BGW} integral},
  Preprint arXiv:2109.01499, 2021.

\bibitem[BCD21b]{us:Maple}
\bysame, \emph{Worksheet accompanying the present paper, available at
  {"\url{http://www.irif.fr/\~chapuy/worksheets/countingMaps-nonoriented.html}"}
  or
  {"\url{http://www.irif.fr/\~chapuy/worksheets/countingMaps-nonoriented.mw}"}},
  2021.

\bibitem[BCEGF21]{BelliardCharbonnierEynardGarciaFailde2021}
R.~Belliard, S.~Charbonnier, B.~Eynard, and E.~Garcia-Failde, \emph{Topological
  recursion for generalised {K}ontsevich graphs and r-spin intersection
  numbers}, Preprint arXiv:2105.08035, 2021.

\bibitem[BDBKS20]{BychkovDuninBarkowskiKazarianShadrin2020}
B.~Bychkov, P.~Dunin-Barkowski, M.~Kazarian, and S.~Shadrin, \emph{{Topological
  recursion for Kadomtsev-Petviashvili tau functions of hypergeometric type}},
  Preprint arXiv:2012.14723, 2020.

\bibitem[BG14]{BouttierGuitter2014}
J.~Bouttier and E.~Guitter, \emph{On irreducible maps and slices},
  Combinatorics, Probability and Computing \textbf{23} (2014), no.~6, 914--972.

\bibitem[BGM21]{BouttierGuitterMiermont2021}
J.~Bouttier, E.~Guitter, and G.~Miermont, \emph{B{ijective enumeration of
  planar bipartite maps with three tight boundaries, or how to slice pairs of
  pants}}, Preprint arXiv:2104.10084, 2021.

\bibitem[BGR08]{BenderGaoRichmond2008}
E.~A. Bender, Z.~Gao, and L.~B. Richmond, \emph{{The map asymptotics constant
  {$t_g$}}}, Electron. J. Combin. \textbf{15} (2008), no.~1, Research paper 51,
  8. \MR{2398843 (2009a:05096)}

\bibitem[BL20]{BudzinskiLouf2020a}
T.~Budzinski and B.~Louf, \emph{Local limits of uniform triangulations in high
  genus}, Inventiones mathematicae \textbf{223} (2020), no.~1, 1--47.

\bibitem[{Car}14]{Carrell2014}
S.~R. {Carrell}, \emph{{The Non-Orientable Map Asymptotics Constant \$p\_g\$}},
  Preprint arXiv:1406.1760, June 2014.

\bibitem[CC15]{CarrellChapuy2015}
S.~R. Carrell and G.~Chapuy, \emph{Simple recurrence formulas to count maps on
  orientable surfaces}, Journal of Combinatorial Theory, Series A \textbf{133}
  (2015), 58--75.

\bibitem[CD17]{ChapuyDolega2017}
G.~Chapuy and M.~Do{\l}{\k{e}}ga, \emph{A bijection for rooted maps on general
  surfaces}, J. Combin. Theory Ser. A \textbf{145} (2017), 252--307.
  \MR{3551653}

\bibitem[CEO06]{ChekhovEynardOrantin2006}
L.~Chekhov, B.~Eynard, and N.~Orantin, \emph{Free energy topological expansion
  for the 2-matrix model}, J. High Energy Phys. (2006), no.~12, 053, 31.
  \MR{2276699}

\bibitem[CMS09]{ChapuyMarcusSchaeffer2009}
G.~Chapuy, M.~Marcus, and G.~Schaeffer, \emph{{A bijection for rooted maps on
  orientable surfaces}}, SIAM Journal on Discrete Mathematics \textbf{23}
  (2009), no.~3, 1587--1611.

\bibitem[DL20]{DolegaLepoutre2020}
M.~Do{\l}{\k{e}}ga and M.~Lepoutre, \emph{Blossoming bijection for bipartite
  pointed maps and parametric rationality of general maps on any surface},
  Preprint arXiv:2002.07238, 2020.

\bibitem[EO07]{EynardOrantin2007}
B.~Eynard and N.~Orantin, \emph{Invariants of algebraic curves and topological
  expansion}, Commun. Number Theory Phys. \textbf{1} (2007), no.~2, 347--452.
  \MR{2346575}

\bibitem[Eyn14]{Eynard2014}
B.~Eynard, \emph{A short overview of the "topological recursion"}, Proceedings
  of the International Congress of Mathematicians---Seoul 2014 \textbf{3}
  (2014), 1063--1085.

\bibitem[Eyn16]{Eynard:book}
B.~Eynard, \emph{Counting surfaces}, Progress in Mathematical Physics, vol.~70,
  Birkh\"auser/Springer, [Cham], 2016, CRM Aisenstadt chair lectures.
  \MR{3468847}

\bibitem[GJ08]{GouldenJackson2008}
I.~P. Goulden and D.~M. Jackson, \emph{The {KP} hierarchy, branched covers, and
  triangulations}, Adv. Math. \textbf{219} (2008), no.~3, 932--951.
  \MR{2442057}

\bibitem[KKN99]{KazakovKostovNekrasov1999}
V.~A. Kazakov, I.~K. Kostov, and N.~Nekrasov, \emph{D-particles, matrix
  integrals and {KP} hierarchy}, Nuclear Physics B \textbf{557} (1999), no.~3,
  413--442.

\bibitem[KRR13]{Kac2013Bombay}
V.~G. Kac, A.~K. Raina, and N.~Rozhkovskaya, vol.~29, World scientific, 2013.

\bibitem[KvdL98]{KacVandeLeur1998}
V.~Kac and J.~van~de Leur, \emph{The geometry of spinors and the multicomponent
  bkp and dkp hierarchies}, CRM Proceedings and Lecture Notes, vol.~14, 1998,
  pp.~159--202.

\bibitem[KZ15]{KazarianZograf2015}
M.~Kazarian and P.~Zograf, \emph{Virasoro constraints and topological recursion
  for {G}rothendieck's dessin counting}, Lett. Math. Phys. \textbf{105} (2015),
  no.~8, 1057--1084. \MR{3366120}

\bibitem[Led09]{Ledoux2009}
M.~Ledoux, \emph{A recursion formula for the moments of the gaussian orthogonal
  ensemble}, Annales de l{\textquotesingle}Institut Henri Poincar{\'{e}},
  Probabilit{\'{e}}s et Statistiques \textbf{45} (2009), no.~3.

\bibitem[Lep19]{Lepoutre2019}
M.~Lepoutre, \emph{Blossoming bijection for higher-genus maps}, Journal of
  Combinatorial Theory, Series A \textbf{165} (2019), 187--224.

\bibitem[Lou19]{Louf2019}
B.~Louf, \emph{Simple formulas for constellations and bipartite maps with
  prescribed degrees}, Canadian Journal of Mathematics \textbf{73} (2019),
  no.~1, 160--176.

\bibitem[MJD00]{MiwaJimboDate2000}
T.~Miwa, M.~Jimbo, and E.~Date, \emph{Solitons}, Cambridge Tracts in
  Mathematics, vol. 135, Cambridge University Press, Cambridge, 2000,
  Differential equations, symmetries and infinite-dimensional algebras,
  Translated from the 1993 Japanese original by Miles Reid. \MR{1736222}

\bibitem[Oko00]{Okounkov2000a}
A.~Okounkov, \emph{Toda equations for {H}urwitz numbers}, Math. Res. Lett.
  \textbf{7} (2000), no.~4, 447--453. \MR{1783622 (2001i:14047)}

\bibitem[Tut62a]{Tutte1962a}
W.~T. Tutte, \emph{{A census of planar triangulations}}, Canad. J. Math
  \textbf{14} (1962), no.~1, 21--38.

\bibitem[Tut62b]{Tutte1962c}
\bysame, \emph{{A census of slicings}}, Canad. J. Math \textbf{14} (1962),
  no.~4, 708--722.

\bibitem[Tut63]{Tutte1963}
\bysame, \emph{{A census of planar maps}}, Canad. J. Math. \textbf{15} (1963),
  249--271. \MR{0146823 (26 \#4343)}

\bibitem[VdL01]{VandeLeur2001}
J.~Van~de Leur, \emph{Matrix integrals and the geometry of spinors}, J.
  Nonlinear Math. Phys. \textbf{8} (2001), no.~2, 288--310. \MR{1839189}

\bibitem[WL72]{WalshLehman1972}
T.~R.~S. Walsh and A.~B. Lehman, \emph{Counting rooted maps by genus. {I}}, J.
  Combinatorial Theory Ser. B \textbf{13} (1972), 192--218. \MR{314686}

\end{thebibliography}
\end{document}